\documentclass[11pt]{amsart}
\usepackage{silence}
\ActivateWarningFilters[mybib]
\usepackage{amsmath}

\DeclareMathAlphabet{\mathup}{OT1}{\familydefault}{m}{n}

\usepackage{amssymb}

\usepackage[utf8]{inputenc}
\usepackage[OT1,T2A]{fontenc}
\usepackage[whole]{bxcjkjatype}
\usepackage{alphabeta}

\usepackage{fixmath}

\usepackage{newunicodechar}
\newunicodechar{“}{``}
\newunicodechar{”}{''}
\newunicodechar{⸨}{\mathopen{(\!(}}
\newunicodechar{⸩}{\mathclose{)\!)}}

\usepackage[main=english,french,german,spanish,italian,russian,greek,japanese,provide*=*]{babel}
\usepackage{csquotes}

\usepackage{amsmath,color}
\let\epsilon\varepsilon
\let\phi\varphi

\let\hat\widehat
\let\tilde\widetilde
\let\bar\overline

\newcommand{\relmiddle}[1]{\mathrel{}\middle#1\mathrel{}}

\newcommand{\ZZ}{\mathbb{Z}}
\newcommand{\CC}{\mathbb{C}}

\newcommand{\AAA}{\mc{A}}
\newcommand{\abs}[1]{\lvert #1 \rvert}
\newcommand{\mc}[1]{\mathcal{#1}}

\newcommand{\MM}{\mc{M}}
\newcommand{\MMM}{\mathfrak{M}}

\newcommand{\switt}{\mathfrak{switt}}

\newcommand{\OO}{\mc{O}}

\def\co{\colon\thinspace}
\usepackage{colonequals}
\def\coeq{\colonequals}

\DeclareMathOperator{\Ad}{Ad}

\DeclareMathOperator{\div1}{div}

\DeclareMathOperator{\D}{D}

\DeclareMathOperator{\vdim}{vdim}
\DeclareMathOperator{\SWitt}{SWitt}

\newcommand{\pp}[1]{\frac{\partial}{\partial{#1}}}

\DeclareMathOperator{\Spec}{Spec}
\DeclareMathOperator{\Spf}{Spf}

\DeclareMathOperator{\Aut}{Aut}

\newcommand{\Gr}{{\mathup{Gr}}}

\DeclareMathOperator{\id}{id}

\DeclareMathOperator{\Hom}{Hom}

\DeclareMathOperator{\Immm}{Im}
\renewcommand{\Im}{\Immm}

\DeclareMathOperator{\Pic}{Pic}

\DeclareMathOperator{\sdim}{sdim}

\DeclareMathOperator{\Ber}{Ber}

\DeclareMathOperator{\str}{str}

\newcommand{\NS}{{\mathup{NS}}}

\newcommand{\red}{\mathup{red}}

\definecolor{mygreen}{RGB}{80,150,10}
\newcommand{\old}[1]{}

\usepackage[style=alphabetic,backend=biber]{biblatex}
\usepackage[linktoc=page,unicode,naturalnames]{hyperref}
\usepackage{xurl}
\hypersetup{breaklinks=true}
\usepackage{amsmath,amssymb,amsthm}
\usepackage{mathtools}
\usepackage{graphicx,import,pdfpages}
\usepackage{tikz}
\usepackage{multicol}
\usepackage{tikz-cd}
\usetikzlibrary{arrows}
\usepackage[normalem]{ulem}
\usepackage{mathtools}
\usepackage{comment}
\usepackage{color} 

\makeatletter
\DeclareRobustCommand*{\tuburl}{\hyper@normalise\tuburl@}
\ExplSyntaxOn
\def\tuburl@#1
{
	\str_set:Nn\l_tmpa_str{#1}
	\str_remove_once:Nn\l_tmpa_str{https://}  
	\str_remove_once:Nn\l_tmpa_str{http://} 
	\expandafter\hyper@linkurl\expandafter{\expandafter\Hurl\expandafter{\l_tmpa_str}}{https://\l_tmpa_str}
}
\ExplSyntaxOff
\makeatother
\usepackage[shortlabels]{enumitem}

\usepackage{soul}

\usepackage[capitalize]{cleveref}
\crefname{equation}{}{}

\usepackage[calc,showdow,en-GB,showseconds=false]{datetime2}
\DTMsettimestyle{iso}

\usepackage{arydshln}

\usepackage[marginratio={1:1}]{geometry}
\usepackage{graphicx}
\usepackage{amssymb,amscd}
\usepackage{epstopdf}

\usepackage[justification=centering]{caption}
\usepackage{float}

\definecolor{lightgray}{rgb}{0.9,0.9,0.9}
\usepackage{xcolor}

\usepackage{colortbl}

\usepackage{hyperref}

\usepackage{cleveref}

\crefformat{section}{\S#2#1#3} 
\crefformat{subsection}{\S#2#1#3}
\crefformat{subsubsection}{\S#2#1#3}
\crefformat{theorem}{Theorem \protect{$#2#1#3$}}
\crefformat{cor}{Corollary \protect{\(#2#1#3\)}}

\usepackage{tikz}

\let\tilde\widetilde

\newtheorem{theorem}{Theorem}[section]
\newtheorem{cor}[theorem]{Corollary}
\newtheorem{prop}[theorem]{Proposition}
\newtheorem{lemma}[theorem]{Lemma}

\theoremstyle{definition}
\newtheorem{definition}[theorem]{Definition}

\theoremstyle{remark}
\newtheorem{rem}[theorem]{Remark}

\hypersetup{colorlinks=true, linkcolor=blue!70!black, urlcolor=blue!70!black, citecolor=blue!70!black,linktoc=page}

\title{The Neveu-Schwarz group and Schwarz's extended super Mumford form}
\author[K. A. Maxwell]{Katherine A. Maxwell}
\email[K. A. Maxwell]{katherine.maxwell@ipmu.jp}
\address{Kavli IPMU (WPI), UTIAS,
    The University of Tokyo,
    Kashiwa, Chiba 277-8583, Japan
}

\author[A. A. Voronov]{Alexander A. Voronov}
\email[A. A. Voronov]{voronov@umn.edu}
\address{School of Mathematics,
  University of Minnesota,
  Minneapolis, MN 55455,
  and\newline
 Kavli IPMU (WPI), UTIAS,
 The University of Tokyo,
 Kashiwa, Chiba 277-8583, Japan
 }
\thanks{Research is supported in part by World Premier International Research Center Initiative (WPI Initiative), MEXT, Japan. The first author is grateful to Max Planck Institute for Mathematics in Bonn for its hospitality and financial support.}

\begin{document}
\begin{abstract}
    In 1987, Albert Schwarz suggested a formula which extends the super Mumford form from the moduli space of super Riemann surfaces into the super Sato Grassmannian. His formula is a remarkably simple combination of super tau functions. We compute the Neveu-Schwarz action on super tau functions, and show that Schwarz's extended Mumford form is invariant under the the super Heisenberg-Neveu-Schwarz action, which strengthens Schwarz's proposal that a locus within the Grassmannian can serve as a universal moduli space with applications to superstring theory. Along the way, we construct the Neveu-Schwarz, super Witt, and super Heisenberg formal groups.
\end{abstract}
\maketitle

\tableofcontents

\section*{Introduction}

The super Mumford form $μ$ is the trivializing section given by the image of $1$ under the canonical isomorphism
\begin{align*}
    \OO_{\mathfrak{M}_g} &= \lambda_{3/2} \otimes \lambda_{1/2}^{-5}, &
    1 &\mapsto μ,
\end{align*}
known as the super Mumford isomorphism, which identifies the structure sheaf with a product of Berezinian line bundles over the moduli space $\mathfrak{M}_g$ of $N=1$ super Riemann surfaces (SRSs) of genus $g$. This super Mumford isomorphism is a generalization of the Mumford isomorphism over the moduli space of Riemann surfaces proved by \textcite{Mumford.1977.sopv} using the Grothendieck-Riemann-Roch theorem. The proof of the super Mumford isomorphism \cite{Deligne.1988., Voronov.1988.afftMmist} provided not only an explicit formula for the super Mumford form, but also showed that the super Mumford isomorphism is canonical, a property lacking for the classical isomorphism.

Prompted by the result of \textcite{Belavin.Knizhnik.1986.cgattoqs} that the Polyakov measure in string theory has a simple explicit relation to the classical Mumford form, algebro-geometric or holomorphic methods became of great interest to computations in (super)string theory. Applying algebro-geometric methods to integrals over the moduli space of SRSs, seen as the ultimate goal, puts the super Mumford form in a place of particular importance. The precise details of the relationship between the superstring measure and the super Mumford form can be found in \cite{Witten.2019.nosRsatm}.

After a relatively dormant period of some 15 years, D'Hoker and Phong made a breakthrough computation of the amplitudes for the supermoduli space $\MMM_2$ of genus-two super Riemann surfaces \cite{DHoker.Phong.2002.tls,DHoker.Phong.2008.lotls,Witten.2015.nohsastmolg} as well as proposed a computation-friendly expression for the super Mumford form on $\MMM_3$ \cite{DHoker.Phong.2005.amfatsm}, which resulted in partial computation of the amplitudes \cite{Gomez.Mafra.2013.tcs3laaSd}. D'Hoker-Phong's computations were based on ``splitting'' the supermoduli space $\MMM_2$ into the underlying moduli space $\MM_2$ and vector-bundle data on it, and then identifying the moduli space $\MM_2$ of Riemann surfaces with the moduli space $\AAA_2$ of principally polarized abelian varieties of dimension 2. Eyeing possible extension of D'Hoker and Phong's results to higher genera, Donagi and Witten \cite{Donagi.Witten.2015.ssinp} showed that the supermoduli space cannot be split. Moreover, since for higher $g$ the moduli space $\MM_g$ is described as a subspace of $\AAA_g$ via complicated equations (see Shiota's solution \cite{Shiota.1986.coJvitose} of the Schottky problem and also Farkas-Grushevsky-Salvati Manni \cite{Farkas.Grushevsky.SalvatiManni.2021.aesttwSp}), a direct generalization of D'Hoker-Phong's computations to higher genera seems to be out of reach, at least for the time being. There is a modular-form ansatz \cite{Grushevsky.2009.ssaihg} for the odd component of the super Mumford form in arbitrary genus, based on a certain splitting assumption for the super Mumford form and subject to verification of physical constraints, such as the vanishing of the cosmological constant.

On the other hand, the moduli space of SRSs admits an embedding into the super Sato Grassmannian. In the seminal paper \cite{Manin.1986.cdotstatdsotmsosc}, Manin conjectured that the moduli space of curves is an orbit of the Virasoro algebra action on the classical Grassmannian, and that a similar statement holds in the super case. The proof of this conjecture in the super case can be found in \textcite{Maxwell.2022.tsMfaSG}. The morphism which embeds supermoduli space in the Grassmannian is known as the super Krichever map \cite{Mulase.Rabin.1991.sKf}.

The embedding of moduli space via the Krichever map lead to proposals that the Sato Grassmannian, or some locus within, is a universal moduli space containing curves of every genus, and similarly in the super case. See for example \cite{Morozov.1987.statsoums,Schwarz.1998.Gast}. Besides the promise that the Sato Grassmannian contains the moduli space in a “universal” way, the relatively simple, although infinite, coordinates of the Grassmannian may provide an accessible setting to perform computations. This idea is supported by fact that Shiota's solution (based on Mulase's work \cite{Mulase.1984.csiseaJv}) to the Schottky problem relies on characterizing the moduli space locus in the Sato Grassmannian via the KP flow. For the super Schottky problem, Mulase described an analogous solution \cite{Mulase.1991.ansKsaacotJoaasc} using the super KP flow.

This paper focuses on the formula suggested by Albert Schwarz \cite{Schwarz.1987.tfsaaums, Schwarz.1989.fsaums} defined using the super KP flow, which is a formula for an extended Mumford form over a certain locus within the super Sato Grassmannian. This locus, referred to as Schwarz's locus, may serve as a universal moduli space as it contains the image of the super moduli space of SRSs under the Krichever map. Schwarz's extended Mumford form is defined as
\begin{align*}
    M(L)\coeq \frac{{\tau}_{L}(g^3)}{{\tau}^3_{L}(g)}, &&  \text{for any $g$ such that }g L = L^\perp,
\end{align*}
where $\tau_L(g)$ is a super tau function, defined for a point $L$ in the Grassmannian and an element $g$ of the super Heisenberg group, which generates the super KP flow. Some properties of these super tau functions were studied in \cite{Dolgikh.Schwarz.1990.sGstfas}.
It is remarkable that the formula is so simple, and also that, as Schwarz noted, a similar formula in the bosonic (classical) case would not make sense, as it would be divergent.

The goal of this paper is to show the invariance of Schwarz's formula under the Neveu-Schwarz (NS) action on the Grassmannian, \cref{NS mumford invariance}.
By restricting to the supermoduli space orbit, this gives a short verification of Schwarz's claim that his formula agrees with the super Mumford form on supermoduli space up to a constant factor. We also give a different proof of the agreement, which removes the constant factor from the claim, \Cref{sMform}. We also present an explicit construction \cref{sMumfordIso} of the super Mumford isomorphism \cite{Voronov.1988.afftMmist} which uses a rational section of the relative dualizing sheaf and thereby does not require the existence of a holomorphic section.

 In order to consider the NS action in relation to the super KP flows, we must consider the NS group. However, it is known that the complex NS group does not exist for the same reason that the complex Virasoro group does not exist. Instead, we define the NS, super Witt, super Heisenberg \emph{formal} groups, inspired by their bosonic counterparts \cite{AlvarezVazquez.MunozPorras.PlazaMartin.1998.tafoseoabf, MunozPorras.PlazaMartin.2001.agok}, and these supergroups suffice to act on the super Sato Grassmannian and describe Schwarz's super tau function. These formal supergroups are of interest in their own right, perhaps in relation to super diffeomorphism groups or super universal Teichmüller space.

\subsection*{Conventions}
We work over the ground field $\CC$ of complex numbers: all super vector spaces, superschemes, etc.\ are assumed to be over $\CC$. By default, we assume our \emph{locally free sheaves} are of constant finite rank and also interchangeably call them \emph{vector bundles}. \emph{Invertible sheaves} or \emph{line bundles} are locally free sheaves of rank $1|0$ or $0|1$. We call them \emph{even} or \emph{odd line bundles}, respectively. We systematically write $=$ for canonical isomorphisms and $\mc{L}^{-1}$ for the dual $\mc{L}^*$ of a line bundle $\mc{L}$.

\section{The superspace of formal Laurent series}

Recall from \cite{Maxwell.2022.tsMfaSG} (see also \Cref{sGrassmannian} below) that the super Sato Grassmannian $\Gr_{j/2}$ parametrizes discrete supersubspaces of
\begin{equation*}
H_{j/2}\linebreak[1] \coeq \linebreak[0] \mathbb{C}(\!(z)\!)[ζ]\;[dz|dζ]^{\otimes j},
\end{equation*}
the space of $j/2$-differentials on the \emph{formal deleted superdisk} $\Spf \mathbb{C}(\!(z)\!)[ζ]$.
This definition differs from the approach of Schwarz, which is to identify $j/2$-differentials with functions by omitting the $[dz|dζ]^{\otimes j}$ factor, as this gives an isomorphism
\begin{equation*}
H_{j/2} \cong Π^j H_{0/2}.
\end{equation*}
The parity change operator $Π$ is due to the convention that $[dz|dζ]$ is odd.
These spaces admit natural multiplications
\begin{align}
\label{H multiplication}
    H_{i/2} \otimes H_{j/2} & \to H_{(i+j)/2},\\\nonumber
    g(z|ζ)\;[dz|dζ]^{\otimes i} \otimes f(z|ζ)\;[dz|dζ]^{\otimes j}& \mapsto \pm g(z|ζ)\cdot f(z|ζ)\; [dz|dζ]^{\otimes (i+j)},
\end{align}
where the sign is given by the Koszul rule, which makes
\begin{equation*}
H_\bullet \coeq
\bigoplus_{j \in \ZZ} H_{j/2}
\end{equation*}
a $\tfrac{1}{2}\ZZ$-graded supercommutative $\CC$-algebra with $[dz|dζ]$ having parity 1 and degree $1/2$. The commutation relation involves only parity.

The space $H_{j/2}$ has a ``semi-infinite polarization,'' that is to say, decomposes into the direct sum
\begin{equation*}
H_{j/2} = H^-_{j/2} \oplus H^+_{j/2},
\end{equation*}
where 
\begin{align*}
H^-_{j/2} & \coeq z^{-1} \CC[z^{-1} | ζ]\;[dz|dζ]^{\otimes j},\\
H^+_{j/2} & \coeq \CC[\![z]\!][ζ]\;[dz|dζ]^{\otimes j}.
\end{align*}
The space $H^+_{j/2}$ will play the role of a \emph{distinguished
compact subspace} of $H_{j/2}$, whereas $H^-_{j/2}$ will serve as a
\emph{distinguished discrete subspace}, making a basepoint of
$\Gr_{j/2}$. However, the super Sato Grassmannian breaks into more
connected components, which will be labeled by virtual dimension,
whence it will be useful to introduce a supply of standard compact and
discrete subspaces for $m,n\in \ZZ$, as follows.
\begin{align}\label{indexed H plus minus}
\left(H^{m|n}_{j/2}\right)^- & \coeq \begin{cases}
    z^{m-1} \CC[z^{-1}]\;[dz|dζ]^{\otimes j} \oplus  z^{n-1} ζ\CC[z^{-1}]\;[dz|dζ]^{\otimes j} & j \text{ even}\\
    z^{m-1} ζ\CC[z^{-1}]\;[dz|dζ]^{\otimes j} \oplus  z^{n-1} \CC[z^{-1}]\;[dz|dζ]^{\otimes j} & j \text{ odd}\\
\end{cases},\\
\left(H^{m|n}_{j/2}\right)^+ & \coeq \begin{cases}
    z^m\CC[\![z]\!]\;[dz|dζ]^{\otimes j} \oplus z^nζ\CC[\![z]\!]\;[dz|dζ]^{\otimes j} & j \text{ even}\\
    z^m ζ \CC[\![z]\!]\;[dz|dζ]^{\otimes j} \oplus z^n\CC[\![z]\!]\;[dz|dζ]^{\otimes j} & j \text{ odd}\\
\end{cases}.\nonumber
\end{align}
Note that the distinguished subspaces correspond to the $0|0$ index as
\begin{align*}
\left(H^{0|0}_{j/2}\right)^- &= H^-_{j/2},\\
\left(H^{0|0}_{j/2}\right)^+ & = H^+_{j/2}.
\end{align*}

\section{The super Heisenberg group}

We would also like to speak about Lie supergroup actions, for which the language of functors of points will be convenient, if not necessary.

There will be two formal supergroups under consideration. Let us start with the simpler one, the super Heisenberg group $Γ_{\bullet}$, which is roughly a certain  supergroup of ``invertible'' elements of $H_\bullet$. A more precise definition via the functor of points may be given by generalizing the definition of the classical formal Heisenberg group from \cite[Section 4]{AlvarezVazquez.MunozPorras.PlazaMartin.1998.tafoseoabf}.
For motivational purposes, let us recall the following result from that paper:
\begin{lemma}[{\cite[Corollary 4.7]{AlvarezVazquez.MunozPorras.PlazaMartin.1998.tafoseoabf}}]
\label{nilpotent-red}
Let $R$ be a $\CC$-algebra and $f(z) = \sum_n a_n z^n \in R(\!(z)\!)$.
Then $f(z)$ is (multiplicatively) invertible if and only if there exists $N \in \ZZ$ such that
$a_n \in \sqrt{R}$ for all $n < N$ and $a_N$ is invertible.
\end{lemma}
Here $\sqrt{R}$ is the nilradical, the set of nilpotent  elements of $R$. A similar statement works in the super case, as we make it clear in the following definition.

\begin{definition}[super Heisenberg group $Γ_{\bullet}$]
\label{Gamma}
Define a functor
\begin{equation*}
Γ_{\bullet}\co \operatorname{\mathbf{SSch}}_\CC  \to \operatorname{\mathbf{Group}}
\end{equation*}
of even invertible elements of $H_\bullet$. More explicitly, for each $j, N \in \ZZ$, 
given a $\CC$-superscheme $S$ with $R = Γ(S, \mathcal{O}_S)$, define the set of $S$-points of $Γ_{j/2}^N$ as
\begin{align}
\nonumber
    Γ_{j/2}^N (S) &\coeq \Big\{ f (z| ζ) [dz|dζ]^{\otimes j} \in H_{j/2} (S), \text{ invertible in } H_\bullet(S) \Big\}
    \\
    &\coeq \begin{multlined}[t][0.8\textwidth]
    \Big\{ \sum_{n \ge -M} (a_n + \alpha_n ζ) z^n [dz|dζ]^{\otimes j} \in (H_{j/2}\mathrel{\hat{\otimes}} R)_{\bar{j}}\mathrel{\Big|}\\
\text{for some } M \in \ZZ, a_N \in R_0^\times, a_n \in \sqrt{R_0} \text{ for $n < N$} \Big\},
    \end{multlined}
     \label{nilpotent}
\end{align}
where $R_0^\times$ is the set of invertible even elements of $R$. The condition of total parity $\bar{j} \in \ZZ/2\ZZ$ of these formal Laurent series is equivalent to the condition $a_n \in R_0$ and $\alpha_n \in R_1$. The group structure on the disjoint union $Γ_{\bullet} (S) \coeq \coprod_j \coprod_N Γ_{j/2}^N (S)$ is given by multiplication of power series. As a group $Γ_\bullet (S)$ is obviously isomorphic to the product $Γ_{0/2}^0 (S) \times \ZZ^2$.
\end{definition}
\begin{prop}
\label{rep-Gamma}
The functor
\begin{align*}
Γ_\bullet \co \operatorname{\mathbf{SSch}}_\CC & \to \operatorname{\mathbf{Group}},\\
S & \mapsto Γ_{\bullet} (S),
\end{align*}
is representable by a formal group superscheme, also denoted by $Γ_{\bullet}$. This supergroup is abelian, as it represents a functor to the subcategory of abelian groups in $\operatorname{\mathbf{Group}}$.
\end{prop}

\begin{proof}
The argument of \cite[Theorem 4.10]{AlvarezVazquez.MunozPorras.PlazaMartin.1998.tafoseoabf} can be easily generalized to the super case. We will show the representability of the functor $Γ_{0/2}^0$ and identify $Γ_\bullet$ as $Γ_{0/2}^0 \times \ZZ^2$.

We start with constructing a formal superscheme and will then prove that it represents the functor.

Consider the algebra $\CC[\dots, x_{-1}, x_0, x_1, \dots | \dots, \xi_{-1}, \xi_0, \xi_1, \dots]$ of polynomials in infinitely many even ($x_i$) and odd ($\xi_i$) variables. For each $M,p  \ge 1$, let
\begin{equation}
\label{IMp}
I_{M,p}\coeq (x_{-1}, \dots, x_{-M})^p + (x_{-M-1}, x_{-M-2}, \dots \, | \, \xi_{-M-1}, \xi_{-M-2}, \dots)
\end{equation}
be the ideal generated by the variables $x_{-M-1}, \linebreak[0] x_{-M-2}, \dots$ and $\xi_{-M-1}, \xi_{-M-2}, \dots$ and the $p$th power of the ideal generated by the variables $x_{-1}, \dots, x_{-M}$. Introduce a linear topology on $\CC[\dots, x_{-1}, x_0, x_1, \dots | \dots, \xi_{-1}, \xi_0, \xi_1, \dots]$ by taking the ideal $I_{M,p}$ as a base of the open neighborhoods of 0. Consider the completion of the polynomial algebra in this topology. This version of the algebra of formal power series in infinitely many variables may also be given by the limit
\begin{multline}
\label{series}
\CC [x_0, x_1, \dots | \; \xi_0, \xi_1, \dots ] \{\!\{ x_{-1}, x_{-2}, \dots | \; \xi_{-1},  \xi_{-2},  \dots\}\!\} \\
\begin{aligned}
& \coeq \varprojlim_{M,p} \CC[\dots, x_{-1}, x_0, x_1, \dots | \dots, \xi_{-1}, \xi_0, \xi_1, \dots] /I_{M,p}\\
 & \cong \varprojlim_{p} \CC[x_{-p},\dots, x_{-1}, x_0, x_1, \dots |\; \xi_{-p}, \dots, \xi_{-1}, \xi_0, \xi_1, \dots] / (x_{-1}, \dots, x_{-p})^p.
\end{aligned}
\end{multline}
The last property implies that $\CC [x_0, x_1, \dots | \; \xi_0, \xi_1, \dots ] \{\!\{ x_{-1}, x_{-2}, \dots | \; \xi_{-1},  \xi_{-2},  \dots\}\!\}$ is an admissible topological ring. The same is true about its localization at $x_0$:
\begin{equation*}
A \coeq \CC [x_0, x_0^{-1}, x_1, x_2, \dots | \; \xi_0, \xi_1, \dots ] \{\!\{ x_{-1}, x_{-2}, \dots | \; \xi_{-1},  \xi_{-2},  \dots\}\!\}
\end{equation*}
Therefore, the formal spectrum $\Spf A$ makes sense. We claim that this is the formal superscheme which represents the functor $Γ_{0/2}^0$. The formal group superscheme structure on $\Spf A$ automatically comes out of the group structure of the functor and is given by multiplication of formal power series.

To prove the claim, observe that, for a complex superscheme $S$ with the algebra of regular functions $R = Γ(S, \OO_S)$ considered with discrete topology, we have
\begin{equation*}
\Hom_{\operatorname{\mathbf{fSSch}}} (S, \Spf A) = \Hom_{\operatorname{\mathbf{top-\CC-sAlg}}} (A, R),
\end{equation*}
where $\operatorname{\mathbf{fSSch}}$ is the category of formal superschemes and $\operatorname{\mathbf{top-\CC-sAlg}}$ is the category of topological $\CC$-superalgebras.
A continuous algebra homomorphism $f\co A \to R$ is equivalent to a collection of even elements $\dots, a_{-1} \coeq f(x_{-1}), a_0 \coeq f(x_0), a_1 \coeq f(x_1), \dots$ and odd elements $\dots, \alpha_{-1} \coeq f(\xi_{-1}), \alpha_{0} \coeq f(\xi_0), \alpha_{1} \coeq f(\xi_1), \dots$ of $R$ such that $a_0 \in R^\times$ and $f(I_{M,p}) = 0$ for some $M$ and $p$. The latter is equivalent to the condition that $a_n = \alpha_n = 0$ for all $n < -M$ and that the ideal $(a_{-1}, \dots, a_{-M})$ is nilpotent, $(a_{-1}, \dots, a_{-M})^p = 0$, which implies that $a_{-1}, \dots, a_{-M}$ are nilpotent. This happens if and only if the series $\sum (a_n + \alpha_n ζ) z^n$ satisfies the conditions \eqref{nilpotent} defining an element of $Γ_{j/2}^N (S)$ with $j = N = 0$.
\end{proof}

We would like to briefly describe the Lie algebra of the super Heisenberg group.

\begin{definition}\label{def of h}
    The \emph{super Heisenberg Lie algebra} $\mathfrak{h}$ is the abelian Lie superalgebra with vector space $H_{0/2}$ with bracket given by the commutator of the natural multiplication \cref{H multiplication}.
\end{definition}

\begin{prop}\label{super heisenberg Lie algebra}
    The super Heisenberg algebra $\mathfrak{h}$ is the Lie algebra of the formal group superscheme $Γ^0_{0/2}$.
\end{prop}
\begin{proof}
    The component $Γ^0_{0/2}$ of the Heisenberg group superscheme may be identified as the sheaf of invertible functions on the formal superdisk $\Spf   \CC[\![z]\!][ζ]$. Therefore the Lie superalgebra of $Γ^0_{0/2}$ is the space of all functions on $\Spf   \CC[\![z]\!][ζ]$, which is clearly given by $H_{0/2}$ with the commutator bracket as in \cref{def of h}.
\end{proof}

\color{black}

\section{The super Witt group}

Next, we consider the super Witt group $\SWitt$ and its central extension the Neveu-Schwarz supergroup $\NS$. Compared with the super Heisenberg group, the super Witt group consisting of superconformal automorphisms of the formal deleted superdisk $\Spf \mathbb{C}(\!(z)\!)[ζ]$ is subtler and, in particular, non-abelian. Even its bosonic prototype, the Witt group, whose central extension is called the Virasoro group, seems to contradict the famous observation of Pressley and Segal \cite{Pressley.Segal.1986.lg} that the Virasoro group does not exist. A remarkable achievement of \cite{AlvarezVazquez.MunozPorras.PlazaMartin.1998.tafoseoabf} is the construction of the Virasoro group as a \textbf{formal} group scheme. This does not contradict the statement that the Virasoro group does not exist as a smooth Lie group. Informally, the formal Virasoro group is a formal extension of Segal's semigroup of annuli \cite{Segal.2004.tdocft} in the ``negative'' direction.

\begin{definition}[The super Witt group $\SWitt$]
\label{SWitt}
Define a functor
\begin{align*}
\SWitt\co \operatorname{\mathbf{SSch}}_\CC & \to \operatorname{\mathbf{Group}},\\
S & \mapsto \Aut^s_{R\operatorname{-Alg}} R(\!(z)\!)[ζ],
\end{align*}
of superconformal automorphisms of the topological $R$-algebra $R(\!(z)\!)[ζ] = H_{0/2} \hat{\otimes} R$, with the group law given by composition and $R = Γ(S, \OO_S)$. By a \emph{superconformal automorphism} we mean an $R$-algebra automorphism $\phi$ that preserves the odd distribution generated by $D_ζ = \frac{\partial}{\partial ζ} + ζ \frac{\partial}{\partial z}$, i.e., $\phi^* D_ζ = F D_ζ$ for some $F \in R(\!(z)\!)[ζ]$. An automorphism of a topological algebra is assumed to be a homeomorphism, i.e., to be continuous with a continuous inverse. The \emph{continuity} of $\phi \in \Aut^s_{R\operatorname{-Alg}} R(\!(z)\!)[ζ]$ in the $z$-adic topology means that for each $M \in \ZZ$ there exists an $N \in \ZZ$ such that $\phi (z^N H_{0/2}^+ \hat \otimes R) \subset z^M H_{0/2}^+ \hat \otimes R$.
\end{definition}

More explicitly, an automorphism is determined by the images of $z$ and $ζ$:
\begin{equation}
\begin{split}
\widehat{z} &= u(z) + ζ ω(z),\\
\widehat{ζ} &= η(z) + ζ v(z),
\end{split}
\label{aut}
\end{equation}
and the superconformality condition may be easily shown \cite[Section 2.1.1]{Witten.2019.nosRsatm} to be equivalent to the equation
\begin{equation}
\label{superconf}
D_ζ \widehat{z} = \widehat{ζ} D_ζ \widehat{ζ},
\end{equation}
which may be rewritten in components as the system
\begin{equation}
\begin{split}
ω  &= η v,\\
u' & = v^2 - η η'.
\end{split}
\label{rels}
\end{equation}

To represent the functor $\SWitt$, we will relate it to the functor 
\begin{align*}
\tilde{Γ^1_{0/2}} \co \operatorname{\mathbf{SSch}}_\CC & \to \operatorname{\mathbf{Set}},\\
S & \mapsto 
\{ (u + ζ η,v) \in Γ^1_{0/2}(S)
\times R_0(\!(z)\!) \; | \;
v^2 = u' + η η' \},
\end{align*}
where $u = u(z) \in R_0(\!(z)\!)$ and $η = η(z) \in R_1(\!(z)\!)$ and $u'$ and $η'$ denote derivatives in $z$.

\begin{theorem}
\label{Gamma-tilde}
    The map
    \begin{align*}
    \Aut^s_{R\operatorname{-Alg}} R(\!(z)\!)[ζ] & \to R(\!(z)\!)[ζ]_0 \times R(\!(z)\!)_0,\\
    (\widehat{z}, \widehat{ζ}) & \mapsto (u + ζ η, v),
    \end{align*}
    induces a natural isomorphism of functors
    \begin{equation}
    \label{psi}
    \psi_S \co \SWitt (S) \to \tilde{Γ^1_{0/2}}(S),
\end{equation}
regarded as functors
\begin{equation*}
\operatorname{\mathbf{SSch}}_\CC \to \operatorname{\mathbf{Set}}.
\end{equation*}
\end{theorem}

\begin{proof}
First off, we need to see that the pair $(u + ζ η, v)$ is in $\tilde{Γ^1_{0/2}}(S)$. We know Equation \eqref{rels} is satisfied for $u, v$, and $η$ coming from a superconformal automorphism. Thus, we need to see that $u + ζ η \in Γ^1_{0/2}(S)$.

Since $z$ is invertible in the ring $R(\!(z)\!)[ζ]$, so is $\hat{z}$ and thereby $u(z)$, which will have the form of \cref{nilpotent-red}.
The series $z + z^2 + \dots$ converges in the $z$-adic topology and the automorphism $(z \,| \,ζ) \mapsto (\hat{z} \, | \, \hat{ζ})$ is continuous by definition, hence the coefficients by the nonpositive powers of $z$ in $u(z)$ must be nilpotent, while \eqref{rels} implies that the same is true about the negative powers of $z$ in $v(z)$.
Moreover, the coefficient by $z$ in $u(z)$ must be invertible, because otherwise the induced automorphism of $R_\red (\!(z)\!)$ would not be injective.
Thus, $u + ζ η \in Γ^1_{0/2}(S)$.

Now it is clear that we have got a well-defined map \eqref{psi} and that it is a natural transformation.

The same equations \eqref{rels} show that an element of $\Im \psi_S \subset \tilde{Γ^1_{0/2}}(S)$ determines the automorphism \eqref{aut} uniquely, i.e., $\psi_S$ is injective.

We claim that $\psi_S$ is surjective, or more concretely, every element $(u(z) + ζ η(z), v(z)) \in \tilde{Γ_{0/2}^1}(S)$ comes from an automorphism \eqref{aut}. Indeed, set $ω \coeq η v$, so that Equations \eqref{rels} hold. Then Equations \eqref{aut} define at least a superconformal endomorphism of the $R$-algebra $R(\!(z)\!)[ζ]$. However, it is in fact an automorphism for the following reason.

Since for an endomorphism to be an automorphism is an open condition on the base superscheme $S$, it is enough to check it on $S_\red$, where \eqref{aut} turns into
\begin{align}
\nonumber
    \hat{z} & = u_\red (z),\\
    \label{red}
    \hat{ζ} & = ζ \cdot v_\red (z)
\end{align}
with $u_\red(z)$ being a formal power series in strictly positive powers of $z$ and an invertible coefficient by $z$ and $v_\red(z)$ being a formal power series with an invertible free term. The power series $u_\red(z)$ has a (compositional) inverse $u^{-1}_\red (\hat{z})$, which may be constructed recursively \cite[Proposition 5.4.1]{Stanley.1999.ec} or using the standard formula which labels the terms of $u^{-1}_\red$ by tree-level Feynman diagrams, cf.\ the Lagrange inversion formula \cite[Section 5.4]{Stanley.1999.ec}. Now, since $v_\red$ is a power series with the degree-zero coefficient invertible, $v_\red$ is multiplicatively invertible and we can solve \eqref{red} for $ζ$, so that the inverse endomorphism of \eqref{red} is given by
\begin{align*}
    z &  = u_\red^{-1}(\hat{z}),\\
    ζ &  = \hat{ζ} \cdot \frac{1}{v_\red(u^{-1}_\red (\hat{z}))}.
\end{align*}
\end{proof}

\begin{cor}
\label{SWitt-rep}
There is a formal group superscheme $\SWitt$ representing the functor $\SWitt$.
\end{cor}

\begin{proof}
As in the proof of \Cref{rep-Gamma}, we will describe a formal affine superscheme representing the functor $\SWitt$ to $\operatorname{\mathbf{Set}}$ and induce a group structure from the group structure on the functor.

Indeed, arguing as in the proof of \Cref{rep-Gamma}, we can see that the functor $\tilde{Γ^1_{0/2}}$, which is naturally isomorphic to $\SWitt$ by \Cref{Gamma-tilde}, is represented by the formal superscheme
\begin{equation*}
\Spf \left (\varprojlim_{M,p} (B \otimes C)   /(I_{M,p} + I_M ) \right),
\end{equation*}
where $I_{M,p}$ is defined in \eqref{IMp} and
\begin{align*}
   B & \coeq \CC [\dots, x_{-1}, x_0, x_1, x_1^{-1}, x_2, x_3, \dots | \; \dots, \xi_{-1}, \xi_{0}, \xi_1, \xi_2, \dots ]  ,\\
   C & \coeq \CC [\dots, y_{-2}, y_{-1}, y_0, y_0^{-1}, y_1, y_2, \dots] ,\\
   I_M & \coeq (u'-v^2 + η η')_M,
\end{align*}
with
\begin{equation*}
u \coeq \sum_{n = -M}^\infty x_n z^n, \qquad
    η \coeq \sum_{n = -M}^\infty \xi_n z^n,\qquad
       v \coeq \sum_{n = -M}^\infty y_n z^n,
\end{equation*}
and the ideal $I_M = (u'-v^2 + η η')_M$ is understood as the ideal generated by the quadratic polynomials in $x_n$'s, $y_n$'s, and $\xi_n$'s equal to the coefficients by various powers of $z$ in the series $u'-v^2 + η η'$. In other words, this series is the generating function for the generators of the ideal $I_M$.
\end{proof}

\begin{cor}
 The natural transformation of functors of points
 \begin{align*}
    \SWitt (S) &\to Γ_{0/2}^1 (S),\\
    (\widehat{z}, \widehat{ζ}) & \mapsto u + ζ η,
\end{align*}
 defines an unramified double covering map of formal superschemes:
 \begin{equation*}
 \SWitt\to Γ_{0/2}^1.
 \end{equation*}
\end{cor}

Now we would like to describe the Lie algebra of the super Witt group $\SWitt$.

\begin{definition}
    The \emph{super Witt Lie algebra} is the Lie superalgebra of superconformal vector fields on the formal deleted superdisk $\Spf \CC(\!(z)\!)[ζ]$ with the superconformal structure given by the odd vector field $D_ζ = \frac{\partial}{\partial ζ} + ζ \frac{\partial}{\partial z}$, i.e.,
    \begin{equation*}
    \switt \coeq \left\{ X = f (z | ζ) \pp{z} + g(z|ζ) \pp{ζ}  \;\middle|\; [X, D_ζ ] = F D_ζ \text{ for some } F \in \CC(\!(z)\!)[ζ]\right\}.
    \end{equation*}
\end{definition}

\begin{prop}[\cite{Manin.1988.NSsadefMs} and also \cite{Maxwell.2022.tsMfaSG}]
    $\switt = \left\{ [h D_ζ, D_ζ] \mid h \in \CC(\!(z)\!)[ζ]\right\}$.
\end{prop}

The following statement, which identifies the Lie algebra of the Lie supergroup $\SWitt$ as $\switt$, fits the general principle ``the Lie algebra of automorphisms of a geometric object preserving a geometric structure is the Lie algebra of those vector fields on the geometric object which preserve this geometric structure.'' However, the specifics of the situation (infinite dimensionality, the super case, formal group schemes rather than Lie groups) obviously require proper justification.

\begin{prop}\label{switt Lie algebra}
    The super Witt algebra $\switt$ is the Lie algebra of the formal group superscheme $\SWitt$. 
\end{prop}

\begin{proof}
The super Witt group $\SWitt$ is a group subsuperscheme of the formal group superscheme of all automorphisms of the formal deleted superdisk $\Spf \CC(\!(z)\!)[ζ]$. Therefore the Lie superalgebra of $\SWitt$ is a subalgebra of the Lie superalgebra of vector fields on $\Spf \CC(\!(z)\!)[ζ]$. Thus, we just need to identify this subalgebra as $\switt$.

An even tangent vector at the identity element of $\SWitt$ is given by \eqref{aut} with all the components $u,v, η, ω\in \CC[\epsilon](\!(z)\!)$ being Laurent series in $z$ over $\CC[\epsilon]$, where $\epsilon$ is an even dual number: $\abs{\epsilon} = 0$, $\epsilon^2 = 0$, such that
\begin{align*}
    \left.\big(u + ζ ω, η + ζ v\big)\right|_{\epsilon = 0} = (z, ζ)
\end{align*}
and $u + ζ ω$ is even and $η + ζ v$ is odd.
Then we must have $ω = η =0$ by a parity argument and $u(z) = z+\epsilon f(z)$, $v(z) = 1$, and $v(z) = 1 + \tfrac{1}{2} \epsilon f(z)$. Thus, an even tangent vector may be written as follows:
\begin{align*}
    \big(z,ζ\big) + \epsilon\big(f,\;\tfrac{1}{2}f'ζ  \big),
\end{align*}
which may be identified with the commutator $f \pp{z} + \tfrac{1}{2}f' ζ \pp{ζ}  =  \tfrac{1}{2}[f D_ζ, D_ζ]$ for $f \in \CC(\!(z)\!)$.

An odd tangent vector at the identity element of $\SWitt$ is given by \eqref{aut} with all the components $u,v, η, ω \in \CC[\delta](\!(z)\!)$ being Laurent series in $z$ over $\CC[\delta]$, where $\delta$ is an odd dual number: $|\delta|=1$, such that
\begin{equation*}
\left.\big(u + ζ ω, η + ζ v\big)\right|_{\delta = 0} = (z, ζ)
\end{equation*}
and $u + ζ ω$ is even and $η + ζ v$ is odd.
Then we must have $u(z) = z$, $v(z) = 1$, and, because of \eqref{rels}, $ω = η $ with $η$ arbitrary, which we may write as $η = \delta g(z) $ for $g \in \CC(\!(z)\!)$, so that the odd tangent vector is written as follows:
\begin{align*}
  \big(z, ζ\big) +  \delta\big(- ζ g ,\;   g\big),
\end{align*}
which may be identified with the commutator $- ζ g \pp{z}+g\pp{ζ}=-\frac{1}{2}[ζ g D_ζ, D_ζ]$.

We can combine these formulas into $[h D_ζ, D_ζ] = \tfrac{1}{2}[f D_ζ, D_ζ] - \frac{1}{2}[ζ g D_ζ, D_ζ]$ for $h = \tfrac{1}{2}f - \tfrac{1}{2}ζ g$. Now note that this is the form of a general element of $\switt$.
\end{proof}

The $\switt$ algebra acts on $j/2$-differentials by Lie derivative \cite{Ueno.Yamada..soogrotsaaasaotMs} and \cite[Section 3.11]{Deligne.Morgan.1999.nos}: an element $[f D_ζ, D_ζ] \in \switt$ acts on $g [dz|dζ]^{\otimes j} \in H_{j/2}$ by the formula
\begin{equation*}
\rho ([f D_ζ, D_ζ]) \; g [dz|dζ]^{\otimes j} \coeq \left( [f D_ζ, \D_ζ] g + \frac{j}{2} \frac{\partial f}{\partial z} g \right) [dz|dζ]^{\otimes j}.
\end{equation*}
This action defines a derivation of the $\CC$-algebra $H_\bullet$.

\section{Actions on \texorpdfstring{$H_\bullet$}{}}\label{actions}

Having defined the formal super Heisenberg and Witt groups, we now describe their action by automorphisms on the super vector space $H_\bullet$ and the induced action on the super Sato Grassmannian.

For a superscheme $S$, let $\hat{(H_{j/2})}_S\coeq H_{j/2}\mathrel{\hat{\otimes}} \mathcal{O}_S$ be the completion of the $\mathcal{O}_S$-module $H_{j/2}\mathrel{\otimes} \mathcal{O}_S$ with respect to the $z$-adic topology. 

The formal super Heisenberg group $Γ_\bullet$ acts on $H_\bullet$ as
\begin{align}\label{gamma H action}
    Γ_{i/2}\times H_{j/2}\to H_{(i+j)/2}
\end{align}
defined by the natural multiplication on $H_\bullet$ as in \eqref{H multiplication}. In particular, in coordinates
\begin{align*}
Γ_{i/2}(S)\times \hat{(H_{j/2})}_S&\to \hat{(H_{(i+j)/2})}_S,\\
    \left(g(z|ζ)\;[dz|dζ]^{\otimes i}, \; f(z|ζ)\;[dz|dζ]^{\otimes j}\right) & \mapsto g(z|ζ) f(z|ζ)\; [dz|dζ]^{\otimes (i+j)}.
\end{align*}
Since the super Heisenberg group is defined as those invertible elements of $H_\bullet$, this action is clearly an automorphism of the super vector space $H_\bullet$.

The formal supergroup $\SWitt$ acts on $H_\bullet$ by pushforward of $j/2$-differentials via automorphisms of the function ring $H_{0/2}$. The action map
\begin{equation}\label{switt H action}
\SWitt \times H_{j/2} \to H_{j/2}
\end{equation}
is defined by the natural transformation of the functors of points
\begin{align*}
\SWitt(S)\times \hat{(H_{j/2})}_S & \to \hat{(H_{j/2})}_S,\\
\left(\phi, \; f(z,ζ)[dz|dζ]^{\otimes j} \right) & \mapsto
 \phi \left(f (z,ζ)[z|dζ]^{\otimes j}\right). \nonumber
\end{align*}
More precisely, if $\phi$ is given by $z \mapsto \widehat{z} = h (z,ζ)$ and $ζ \mapsto \widehat{ζ} = χ(z, ζ)$, then
\begin{equation*}
\phi \left( f(z,ζ)[dz|dζ]^{\otimes j} \right) \coeq f\big(h(z, ζ), χ (z, ζ)\big) [dh(z,ζ)|dχ(z,ζ)]^{\otimes j} .
\end{equation*}
The right-hand side is a $j/2$-differential, because by definition, the supergroup $\SWitt$ preserves the superconformal structure on the formal deleted superdisk, \emph{i.e}., the formal distribution generated by $D_ζ = \pp{ζ} + ζ \pp{z}$. We may think of this distribution as the space of $(-1/2)$-differentials, and therefore, $\SWitt$ maps a $j/2$-differential to a $j/2$-differential.

\begin{rem}
\label{auts}
The action \eqref{switt H action} is an action by automorphisms of the $\CC$-algebra $H_\bullet$.
\end{rem}

The super Heisenberg group inherits a natural action by automorphisms of $H_\bullet$, since $Γ_\bullet$ consists of the invertible elements of $H_\bullet$, see Definition \eqref{Gamma}. In particular, the action is
\begin{align}\label{switt on gamma action}
    \SWitt\times Γ_{j/2}\to Γ_{j/2},
\end{align}
defined by the change-of-variable action on $j/2$-differentials, see \eqref{switt H action}.
 The resulting $j/2$-differential belongs to $Γ_{j/2} (S)$ because a nowhere vanishing section maps to a nowhere vanishing section under an algebra automorphism.

\begin{lemma}\label{equivariant gamma action}
The action \eqref{gamma H action} of the supergroup $Γ_\bullet$ on the super vector space $H_\bullet$,
\begin{equation*}
Γ_\bullet \times  H_\bullet \to H_\bullet,
\end{equation*}
is $\SWitt$-equivariant. Here $\SWitt$ acts diagonally on $Γ_\bullet \times  H_\bullet$, combining the actions \eqref{switt H action} and \eqref{switt on gamma action}.
\end{lemma}

\begin{proof}
The fact that $\SWitt$ acts on $H_\bullet$ via \eqref{switt H action} by automorphisms of the $\CC$-algebra $H_\bullet$, see \cref{auts}, means, in particular, that the multiplication map
\begin{equation*}
H_\bullet \otimes H_\bullet \to H_\bullet 
\end{equation*}
is $\SWitt$-equivariant. The action
\begin{equation*}
Γ_\bullet \times H_\bullet \to H_\bullet
\end{equation*}
was the restriction of the multiplication from $ H_\bullet$ to $Γ_\bullet \subset H_\bullet$, whereas the action \eqref{switt on gamma action} of $\SWitt$ was also obtained by restricting \eqref{switt H action} from $H_\bullet$ to $Γ_\bullet$. This implies the lemma.
\end{proof}

Since the supergroup $\SWitt$ acts on $H_\bullet$ by algebra automorphisms, it automatically acts on the supergroup $Γ_\bullet$ of invertible elements of $H_\bullet$. Using this action, we define the semidirect product of these groups:
\begin{equation*}
Γ_\bullet \rtimes \SWitt.
\end{equation*}

\begin{prop}\label{semidirect Lie algebra}
    The semidirect product $\mathfrak{h} \rtimes\switt$ of the super Heisenberg and super Witt algebra Lie algebras is the Lie algebra of the formal group superscheme $Γ^0_{0/2} \rtimes \SWitt$.
\end{prop}
\begin{proof}
    In \cref{super heisenberg Lie algebra} and \cref{switt Lie algebra}, we showed respectively that the Lie algebra of $Γ^0_{0/2}$ is the abelian superalgebra $\mathfrak{h}$ and that the Lie algebra of $\SWitt$ is the usual $\mathfrak{switt}$.

    It suffices to notice that the action of $\SWitt$ by automorphisms on $Γ^0_{0/2}$ differentiates into the action of $\switt$ by derivatives on $H_{0/2}$.
\end{proof}

\color{black}
\section{The super Sato Grassmannian}

We define the super Sato Grassmannian, for which more details can be found in \cite{Maxwell.2022.tsMfaSG}, before describing the action of the formal supergroups $Γ$ and $\SWitt$ and the duality map.

Given a morphism $T \to S$ of superschemes, for any $\mathcal{O}_S$-submodule $V\subset (\hat{H_{j/2}})_S$, we use the notation below to denote base change via a morphism $T \to S$ and formal completion:
\begin{align*}
    V_T\coeq V\otimes_{\mathcal{O}_S} \mathcal{O}_T, && \hat{V}_T\coeq V \mathrel{\hat{\otimes}_{\mathcal{O}_S}}
    \mathcal{O}_T \coeq \varprojlim_n (V_T/(V_T \cap z^n (H^+_{j/2})_T ).
\end{align*}


\begin{definition}
Define a subspace $K$ of $H_{j/2}$ to be \emph{compact} if it is commensurable with $H^+_{j/2}$. Subspaces $K$ and $H^+_{j/2}$ are \emph{commensurable} when 
\begin{align*}
    \left(H^+_{j/2} + K\right)\Big/\left(H^+_{j/2} \cap K\right)
\end{align*}
is finite dimensional.

\end{definition}

\begin{definition}[super Sato Grassmannian $\Gr(H_\bullet)$]
\label{sGrassmannian}
    Define a functor
    \begin{align*}
        \Gr(H_{j/2})\co \operatorname{\mathbf{SSch}}_\CC & \to \operatorname{\mathbf{Set}}\\
        S&\mapsto \left\{\text{discrete quasi-coherent $\mathcal{O}_S$-submodules $L\subset (\hat{H_{j/2}})_S$}\right\},
    \end{align*} 
    where $L\subset (\hat{H_{j/2}})_S$ is \emph{discrete} if 
    for every $s\in S$ there exists a neighborhood $U$ of $s$ and a compact $K\subset H_{j/2}$ such that the natural map $L_U\oplus \hat{K}_U\to (\hat{H_{j/2}})_U$ is an isomorphism.
\end{definition}

The functor $\Gr(H_{j/2})$ is representable by an infinite dimensional superscheme, which we call the super Sato Grassmannian \cite{Maxwell.2022.tsMfaSG}. In what follows, we use the notation $\Gr_\bullet\coeq \Gr(H_\bullet)$ and $\Gr_{j/2}\coeq \Gr(H_{j/2})$.

One can observe, analogously to the classical case \cite{AlvarezVazquez.MunozPorras.PlazaMartin.1998.tafoseoabf}, the following facts:
\begin{itemize}
    
\item
An $\OO_S$-submodule $L\subset (\hat{H_{j/2}})_S$ is discrete if and only if for each $s \in S$ there exists a neighborhood $U$ of $s$ and an $N \in \ZZ$ such that $L_U\cap z^N \hat{(H^+_{j/2})}_U$ is locally free of finite type and $L_U + z^N \hat{(H^+_{j/2})}_U = (\hat{H_{j/2}})_U$;

\item The above property automatically holds for any $M \le N$;

\item
For each $L \in \Gr_{j/2}(S)$, the complex 
\begin{equation*}
0 \to L \oplus (\hat{H^+_{j/2}})_S\to (\hat{H_{j/2}})_S \to 0
\end{equation*}
is \emph{perfect}, i.e., it is
locally quasi-isomorphic to a complex of finite free
$\OO_S$-modules.
\end{itemize}

The super Sato Grassmannian $\Gr_{j/2}$ decomposes into a disjoint union of connected components corresponding to the virtual dimension of the discrete subspaces making $\Gr_{j/2}$ up:
\begin{equation}
\label{Gr-component}
\Gr_{j/2} = \coprod_{m|n \in \ZZ \times \ZZ} \Gr_{j/2}(m|n),
\end{equation}
where $\Gr_{j/2}(m|n) = \Gr (m|n, H_{j/2})$ is the superscheme based on the connected component $\Gr_{j/2}(m|n)_\red$ of the underlying space $(\Gr_{j/2})_\red$
\begin{equation*}
\Gr_{j/2}(m|n)_\red \coeq
\left\{\text{discrete subspaces $L\subset H_{j/2}$ of $\vdim L = m|n$}\right\},
\end{equation*}
where the \emph{virtual dimension} is defined as
\begin{equation*}
\vdim L \coeq \dim (L \cap H_{j/2}^+) - \dim (H_{j/2}/(L + H_{j/2}^+)).
\end{equation*}
This is, by definition, the \emph{Fredholm index} of the operator $L
\oplus H_{j/2}^+ \to H_{j/2}$.

\section{Actions on the super Sato Grassmannian}

In this section, we show that the actions of \Cref{actions} on $H_\bullet$ preserve the discreteness property of a subspace, and therefore induce actions on the super Sato Grassmannian.

\begin{prop}\label{gamma action on Gr}
   The action \eqref{gamma H action} of the formal super Heisenberg group $Γ_\bullet$ on $H_\bullet$ induces an action on the super Sato Grassmannian $\Gr_\bullet$, additive in the degree of the differentials:
   \begin{align*}
        Γ_{i/2}\times \Gr_{j/2}\to \Gr_{(i+j)/2}.
    \end{align*}
\end{prop}
\begin{proof}
  We need to define a morphism $Γ_\bullet \times  \Gr_\bullet  \to \Gr_\bullet$. We can define it as a natural transformation between functors of points.

      Since an $S$-point $g$ of the super Heisenberg group is an invertible element of $
    (H_{i/2}\mathrel{\hat{\otimes}} R)_{\bar{i}}$ where $R=Γ(S,\mathcal{O}_S)$, then its action by multiplication on the $\mathcal{O}_S$-module $\hat{(H_\bullet)}_S$ is an automorphism.

      Let $L\subset \hat{(H_\bullet)}_S$ represent an $S$-point of $\Gr_\bullet$, and so there exists a small enough neighborhood $U$ on $S$ and a compact $z^N H_{j/2}^+$ such that $L_U\cap z^N\hat{(H_{j/2}^+)}_U$ is locally free of finite type and $L_U\oplus z^N\hat{(H_{j/2}^+)}_U = \hat{(H_{j/2})}_U$. Then $(g L_U) \cap (g z^N \hat{(H^+_{j/2})}_U)$ is locally free of finite type and $(gL_U) + (gz^N \hat{(H^+_{j/2})}_U)$ is still the whole $(\hat{H_{j/2}})_U$. Choose $M \in \ZZ$ such that $g z^N \in z^M \hat{(H^+_{j/2})}_U$, so then $g z^N \hat{(H^+_{j/2})}_U
      \subset z^M \hat{(H^+_{j/2})}_U$. And further $(gL_U) + (z^M \hat{(H^+_{j/2})}_U) = \hat{({H_{j/2}})}_U$. We claim that
      $(g L_U) \cap (z^M \hat{(H^+_{j/2})}_U)$ is locally free of finite type. This will imply that $gL$ represents an $S$-point of $\Gr_{j/2}$.

To justify the claim, first consider the morphism of complexes of sheaves below.
      \begin{equation*}
          \begin{tikzcd}
              0\arrow{r}& gz^N \hat{(H^+_{j/2})}_U \arrow{r}\arrow{d} & z^M \hat{(H^+_{j/2})}_U\arrow{r}\arrow{d} & z^M \hat{(H^+_{j/2})}_U\big/ g z^N \hat{(H^+_{j/2})}_U \arrow{r}\arrow{d} & 0\\
              0\arrow{r}& \hat{(H_{j/2})}_U/L_U \arrow{r} & \hat{(H_{j/2})}_U/L_U\arrow{r} & 0 \arrow{r} & 0
          \end{tikzcd}
      \end{equation*}

      Applying the snake lemma, the resulting LES is in fact a SES in the three kernels. Therefore, to show the claim, it suffices to show that the quotient $z^M \hat{(H^+_{j/2})}_U\big/ g z^N \hat{(H^+_{j/2})}_U$ is locally free of finite type.
      
      For each point $s$ of $U$, the value of $g$ at $s$ will be represented by a Laurent series
      \begin{equation*}
      g(s) = \left(\sum a_{K} z^K + a_{K+1} z^{K+1} + \dots \right) [dz | dζ]^i
      \end{equation*}
        with $K \in \ZZ$, $a_K \in \CC^*$, and $a_{K+1}, a_{K+2}, \ldots \in \CC$. Therefore, $g(s) z^N$ has leading order $N+K$.
By Nakayama's lemma, there is a neighborhood $V$ of $s$ over which $$\big\{z^M, \ldots, z^{(N+K-1)}\mathop{\big|}ζ z^M, \ldots, ζ z^{(N+K-1)}\big\}$$ generate the quotient $z^M \hat{(H^+_{j/2})}_V\big/ g z^N \hat{(H^+_{j/2})}_V$. Since these elements represent a basis of the quotient space at the point $s$, they will also be linearly independent in the neighborhood. Thus, the quotient will be free of finite type over this neighborhood.
\end{proof}

\begin{prop}\label{SWitt action on Gr}
   The action \eqref{switt H action} of the formal super Witt group $\SWitt$ on $H_\bullet$ induces an action on the super Sato Grassmannian $\Gr_\bullet$, preserving the degree of the differentials:
   \begin{align*}
    \SWitt\times \Gr_{j/2}\to \Gr_{j/2}.
\end{align*}
\end{prop}
\begin{proof}
An $S$-point $\phi$ of the super Witt group is an automorphism of the graded $R = H^0(S, \OO_S)$-algebra $ H_\bullet \hat{\otimes} R$ and, in particular, of the $R$-module $ H_{j/2} \hat{\otimes} R$. It induces an automorphism of the trivial vector bundle $\hat{(H_{j/2})}_S$. If $L \subset \hat{(H_{j/2})}_S$ represents an $S$-point of $\Gr_{j/2}$, then in a small enough neighborhood $U$ on $S$, there is an $N \in \ZZ$ such that $L_U\cap z^N \hat{(H^+_{j/2})}_U$ is locally free of finite type and $L_U + z^N \hat{(H^+_{j/2})}_U = (\hat{H_{j/2}})_U$. Then $\phi(L_U) \cap \phi(z^N \hat{(H^+_{j/2})}_U)$ will also be locally free of finite type and $\phi(L_U) + \phi(z^N \hat{(H^+_{j/2})}_U)$ will still be the whole $(\hat{H_{j/2}})_U$. By the continuity of $\phi$, we can find $M \in \ZZ$ such that $\phi(z^N \hat{(H^+_{j/2})}_U) \subset z^M \hat{(H^+_{j/2})}_U$. The fact that the quotient is locally free of finite rank may be shown as in the proof of the previous proposition. Therefore, $\phi L$ will also represent an $S$-point of $\Gr_{j/2}$.
Varying $S$, we get a natural transformation:
\begin{equation*}
(\SWitt \times \Gr_{j/2})(S) \to \Gr_{j/2}(S),
\end{equation*}
which gives the required action of $\SWitt$ on $\Gr_{j/2}$.
\end{proof}

The previous two propositions combine into an action of the semidirect product $Γ_\bullet \rtimes \SWitt$ on the super Sato Grassmannian $\Gr_\bullet$, because the actions on the Grassmannian are induced by the actions on $H_\bullet$ and the action of $Γ_\bullet$ on $H_\bullet$ is $\SWitt$-equivariant by \Cref{equivariant gamma action}.
\begin{cor}\label{semidirect action on Gr}
   The actions of \Cref{gamma action on Gr} and \Cref{SWitt action on Gr} induce an action of the semidirect product $Γ_\bullet \rtimes \SWitt$ on the super Sato Grassmannian $\Gr_\bullet$.
\end{cor}

\section{The Berezinian line bundle}

We recall here the definition of the Berezinian line bundle on the super Sato Grassmannian, see \cite{Maxwell.2022.tsMfaSG}.
We will later describe actions on the Berezinian line bundles and naturally arrive at the construction of the formal Neveu-Schwarz group, following the ideas of \cite{MunozPorras.PlazaMartin.2001.agok} in the classical, Virasoro case.

Let $\mc{L}$ be the \emph{tautological sheaf} over $\Gr_{j/2}$. This is the universal $\OO_\Gr$-module corresponding to the identity $\Gr_{j/2}$-point $\id\co \Gr_{j/2} \to \Gr_{j/2}$ of $\Gr_{j/2}$.

\begin{definition}
\label{Ber-l-b}
    The \emph{Berezinian line bundle} $\Ber = \Ber_{j/2}$ on $\Gr_{j/2}$ is
    the Berezinian of the perfect complex defined by the
    addition morphism:
    \begin{align*}
    \Ber_{j/2} (\mc{L}) \coeq \Ber\Big(0 \to \mc{L}\oplus \hat{\left(H^+_{j/2}\right)}_\Gr \to \hat{(H_{j/2})}_\Gr \to 0 \Big).
    \end{align*}
\end{definition}

Note that for the classical Sato Grassmannian (and finite classical
Grassmannians) the determinant line bundle has no global sections,
while the dual determinant line bundle does have global sections;
however, for the super Sato Grassmannian (and finite super
Grassmannians \cite[Chapter 4 Section 3]{Manin.1988.gftacg}), neither the Berezinian nor the dual Berezinian line
bundles have global sections \cite[Section 3.1]{Bergvelt.Rabin.1999.stJasKe}.

Later, when considering group actions on the Berezinian line bundle, it becomes necessary to consider the Berezinian line bundles defined using the parity reversed perfect complex. Considering $\mathcal{L}$ over $\Gr_{j/2}(m|n)$, then the parity reversed perfect complex is canonically isomorphic to the dual of the usual Berezinian line bundle as:
\begin{align}
    \label{Ber of Pi complex}
    \Ber\Big(0 \to Π^j\mc{L}\oplus Π^jH^+_{j/2} \to Π^jH_{j/2} \to 0 \Big)
    =Π^{m-n}\Ber^{*}_{H_{j/2}^+} (\mc{L}).
\end{align}
This isomorphism is due to the fact that $\Ber Π V = Π^{m-n} (\Ber V)^{-1}$ for a super vector space $V$ of finite dimension $(m | n)$ or, more generally, a perfect complex of Euler characteristic $(m | n)$.

\color{black}

\section{Actions on the Berezinian line bundle} \label{action on Ber section}

Using the action of $\SWitt$ on $Γ_\bullet$, we defined in \Cref{actions} the semidirect product of these groups:
\begin{equation*}
Γ_\bullet \rtimes \SWitt.
\end{equation*}
We would like to define a central extension of the supergroup $Γ_\bullet \rtimes \SWitt$ by lifting its natural action from the Grassmannian $\Gr_\bullet$ to the Berezinian line bundle.

\begin{lemma}\label{Ber preserved}
    The Berezinian line bundle is preserved, up to parity reversal and inversion, under the action of the semidirect product group $Γ_\bullet \rtimes \SWitt$ on the super Sato Grassmannian. To be more precise, for $χ\in Γ_{i/2} \rtimes \SWitt$, we have
    \begin{equation}
        \label{chi-on-Ber}
        Π^{(m-n)(i+j)}  χ^*\Ber_{(i+j)/2}^{(-1)^{i+j}} = Π^{(m-n)j} \Ber_{j/2}^{(-1)^{j}}
    \end{equation}
    in $\Pic (\Gr_{j/2}(m|n))$.
\end{lemma}
\begin{proof}
    More precisely, we want to show that, for any superscheme $S$, each $S$-point $χ \in (Γ_{i/2} \rtimes \SWitt)(S)$ of $Γ_{i/2} \rtimes \SWitt$, acting on the relative Grassmannian $\Gr_\bullet \times S \to S$, gives the equality \eqref{chi-on-Ber} of the corresponding $S$-points
    of the Picard superscheme $\Pic_\Gr$. The $S$-point of $\Pic_\Gr$ corresponding to the Berezinian line bundle is the class of $p_\Gr^* \Ber_{\bullet}$ in $\Pic_\Gr(S) = \Pic(\Gr_{\bullet} \times S)/p_S^* \Pic (S)$, where $p_\Gr$ and $p_S$ denote projection onto the respective factors of $\Gr_\bullet \times S$. We need to show that for each $j$,
    \begin{equation*}
    Π^{(m-n)(i+j)} χ^* p_\Gr^* \Ber_{(i+j)/2}^{(-1)^{i+j}} \cong Π^{(m-n)j} p_\Gr^* \Ber_{j/2}^{(-1)^{j}} \mathop{\otimes} p_S^* \mc{M}
    \end{equation*}
    for some line bundle $\mc{M}$ over $S$.

    Take $K \ge 0$ large enough such that
    $χ \hat{\left(z^K H^+_{j/2}\right)}_S \subset \hat{\left(H^+_{(i+j)/2}\right)}_S$ and
    the quotient
    $\left. \hat{\left(H^+_{(i+j)/2}\right)}_S \middle/ χ\hat{\left(z^K H^+_{j/2}\right)}_S\right.$ is locally free of finite type.
    As in the proof of \Cref{SWitt action on Gr}, this is possible because of the continuity of $χ$, see \cref{SWitt} for $\SWitt$ and note that an $S$-point of $Γ_\bullet$ has finitely many negative powers of $z$ and therefore acts continuously.

    It is natural to lift the action of $χ \in Γ_{i/2} \rtimes \SWitt$ from $\Gr_\bullet$ to the Berezinian line bundle from the isomorphism $ Π^j χ^{-1} Π^{i+j}$ of perfect complexes, which we may denote as $\Ber Π^j χ^{-1} Π^{i+j}$.\footnote{When $i$ is odd, the map $χ$ is a parity reversing isomorphism, which means the Berezinian of $χ$ does not
make sense. To resolve this issue, we consider only parity preserving morphisms by working with the perfect complexes which are
parity reversed when their grading is odd.} This isomorphism induces the following short exact sequence of perfect complexes:
    \begin{equation}\label{Ber chi ses}
            \begin{tikzcd}
                &  0 \arrow{d} & 0 \arrow{d}\\
                0 \arrow{r} & Π^{i+j}χ \mathcal{L} \oplus Π^{i+j}
                H_{(i+j)/2}^+ \arrow{r} \arrow{d}{Π^{j}χ^{-1}Π^{i+j}}& Π^{i+j}H_{(i+j)/2}
\arrow{r}\arrow{d}{Π^{j}χ^{-1}Π^{i+j}}& 0\\
0 \arrow{r} & Π^{j}\mathcal{L} \oplus Π^{j}χ^{-1}H_{(i+j)/2}^+
\arrow{r}\arrow{d} & Π^{j}H_{j/2}
\arrow{d} \arrow{r} & 0\\
0  \arrow{r} & \frac{Π^{j} \left(χ^{-1} H_{j/2}^+\middle/z^K H_{j/2}^+\right)}{ Π^{j}\left(H_{j/2}^+\middle/ z^K H_{j/2}^+\right) } \arrow{r}\arrow{d} & 0 \arrow{r}\arrow{d} & 0\\
&  0  & 0,
            \end{tikzcd}
        \end{equation}
        where $\mathcal{L}$ is the tautological sheaf, and subscripts $\Gr\times S$ have been omitted in the diagram. Rewriting the diagram using the identification with the Berezinian line bundle by applying \cref{Ber of Pi complex} gives
    \begin{multline*}
        Π^{(m-n)(i+j)}χ^* p_\Gr^* \Ber_{(i+j)/2}^{(-1)^{i+j}}(\mathcal{L})\\
        \begin{aligned}
            &= \Ber \left(0 \to χ^* p_\Gr^* Π^{i+j}\mc{L}\oplus Π^{i+j}\hat{\left(H^+_{(i+j)/2}\right)}_{\Gr \times S} \to Π^{i+j}\hat{\left(H_{(i+j)/2}\right)}_{\Gr \times S} \to 0 \right)\\
            &\begin{multlined}
                \xrightarrow[\widetilde{\phantom{xyz}}]{\Ber Π^j χ^{-1} Π^{i+j}}\\
               \Ber\left(0 \to p_\Gr^* Π^j\mc{L}\oplus
                Π^j χ^{-1} \hat{\left(H^+_{(i+j)/2}\right)}_{\Gr \times S}
                \to Π^j χ^{-1} \hat{\left(H_{(i+j)/2}\right)}_{\Gr \times S} \to 0 \right)
            \end{multlined}\\
            & \begin{multlined}
                = \Ber\left(0 \to p_\Gr^* Π^j\mc{L}\oplus
                Π^j \hat{\left(H^+_{j/2}\right)}_{\Gr \times S}
                \to Π^j χ^{-1} \hat{\left(H_{(i+j)/2}\right)}_{\Gr \times S} \to 0 \right)\\
                \otimes \Ber \left( Π^j χ^{-1} \hat{\left(H^+_{(i+j)/2}\right)}_{\Gr \times S} \middle/ Π^j \hat{\left(z^K H^+_{j/2}\right)}_{\Gr \times S} \right)\\
                \otimes \Ber^* \left( Π^j\hat{\left(H^+_{j/2}\right)}_{\Gr \times S} \middle/ Π^j\hat{\left(z^K H^+_{j/2}\right)}_{\Gr \times S} \right),
            \end{multlined}\\
            &
            \begin{multlined}
                =
                Π^{(m-n)j} p_{\Gr}^*\Ber_{j/2}^{(-1)^j} (\mathcal{L})
                \otimes \Ber^{(-1)^j} \left( χ^{-1} \hat{\left(H^+_{(i+j)/2}\right)}_{\Gr \times S} \middle/ \hat{\left(z^K H^+_{j/2}\right)}_{\Gr \times S} \right)\\
                \otimes \Ber^{(-1)^{j+1}} \left( \hat{\left(H^+_{j/2}\right)}_{\Gr \times S} \middle/ \hat{\left(z^K H^+_{j/2}\right)}_{\Gr \times S} \right).
            \end{multlined}
        \end{aligned}
    \end{multline*}
    Observe that the second tensor factor comes from a line bundle $\mc{M}$ over $S$, because
    \begin{equation*}
    \left.χ^{-1} \hat{\left(H^+_{(i+j)/2}\right)}_{\Gr \times S} \middle/ \hat{\left(z^K H^+_{j/2}\right)}_{\Gr  \times S} \right. = p_S^* \left(χ^{-1} \hat{\left(H^+_{(i+j)/2}\right)}_S \middle/ \hat{\left(z^K H^+_{j/2}\right)}_S \right).
    \end{equation*}
    The third tensor factor is similar.
\end{proof}

Before considering the action of $Γ_\bullet\rtimes \SWitt$ on the Berezinian line bundle, we consider the action of the discrete group of shift isomorphisms. The shift operators are $Z^A_{i/2}\coeq z^A[dz|dζ]^{\otimes i}\in Γ_{\bullet}$, for $i,A\in \mathbb{Z}$. They act on $H_\bullet$ as
\begin{multline}\label{shifts primed ZA}
    Z^A_{i/2} \co \left(H_{j/2}^{m|n}\right)^-\oplus \left(H_{j/2}^{m|n}\right)^+ \to \left(H_{(i+j)/2}^{m'|n'}\right)^-\oplus\left(H_{(i+j)/2}^{m'|n'}\right)^+ \\
    \text{where}\quad
    m'|n' \coeq \begin{cases}
        m+A | n+A, & \text{$i$ even,}\\
        n+A | m+A, & \text{$i$ odd,}
    \end{cases}
\end{multline}
inducing isomorphisms between the respective components of the Grassmannian.

\begin{lemma}\label{shift action on Gr}

    The discrete group consisting of the shift isomorphisms
    \begin{align*}
        \mathbb{Z}^2 \cong \left\langle  Z^A_{i/2} \relmiddle| i,A\in \mathbb{Z}\right\rangle
    \end{align*}
    acts on the Berezinian line bundle. Explicitly, for $L\in \Gr_{j/2}^{m|n}$ and $χ=Z^A_{i/2}$, the induced identification of fibers is
    \begin{align*}
       χ\co  Π^{(m-n)j}\Ber_{j/2}^{(-1)^j}(L) \xrightarrow{\sim} Π^{(m-n)(i+j)}\Ber_{(i+j)/2}^{(-1)^{i+j}}(χ L).
    \end{align*}
\end{lemma}

\begin{proof}
    From  the proof of \cref{Ber preserved}:
    \begin{multline*}
        Π^{(m-n)(i+j)}χ^* p_\Gr^* \Ber_{(i+j)/2}^{(-1)^{i+j}}(\mathcal{L})  \xrightarrow[\widetilde{\phantom{xyz}}]{\Ber Π^j χ^{-1} Π^{i+j}}\\
        \begin{aligned}
            &\begin{multlined}
               \Ber\left(0 \to p_\Gr^* Π^j\mc{L}\oplus
                Π^j χ^{-1} \hat{\left(H^+_{(i+j)/2}\right)}_{\Gr \times S}
                \to Π^j χ^{-1} \hat{\left(H_{(i+j)/2}\right)}_{\Gr \times S} \to 0 \right)
            \end{multlined}\\
            &
            \begin{multlined}
                =
                Π^{(m-n)j} p_{\Gr}^*\Ber_{j/2}^{(-1)^j} (\mathcal{L})
                \otimes \Ber^{(-1)^j} p_S^*\left(  \hat{\left(z^{-A}H_{j/2}^+\right)}_{S} \middle/ \hat{\left(z^K H^+_{j/2}\right)}_{ S} \right)\\
                \otimes \Ber^{(-1)^{j+1}} p_S^*\left( \hat{\left(H^+_{j/2}\right)}_{S} \middle/ \hat{\left(z^K H^+_{j/2}\right)}_{ S} \right)
            \end{multlined}
        \end{aligned}
    \end{multline*}
    where the tensor factors combine into the trivial bundle over $S$, up to a parity factor $Π^A$.

    Since the tensor factors define the trivial bundle $p_S^*\mathcal{O}$, the isomorphism above upgrades to a isomorphism between fibers
    \begin{align*}
    	\Ber^{-1}(Π^{i+j} χ^{-1} Π^j)\co Π^{(m-n)j}\Ber_{j/2}^{(-1)^j}(L) \xrightarrow{\sim} Π^{(m-n)(i+j)}\Ber_{(i+j)/2}^{(-1)^{i+j}}(χ L)
    \end{align*}
    which we may denote with simply $χ$. The fact that this isomorphism of fibers is a group action follows from noticing there is no projective multiplier coming from the tensor factors.
\end{proof}

\begin{theorem}
\label{canonical central extension}
There exists a canonical central extension of the semidirect product group
\begin{align*}
    0 \to
    \left(\mathbb{C}^*\right)^{\ZZ} \to  \tilde{Γ_\bullet \rtimes \SWitt}
    \to Γ_\bullet \rtimes \SWitt \to 0,
\end{align*}
and a natural action of the extension, for $\Ber_{j/2}\in \Pic(\Gr_{j/2}(m|n))$, given by
\begin{equation*}
\tilde{\left(Γ_{i/2} \rtimes \SWitt \right)} \times Π^{(m-n)j} \Ber_{j/2}^{(-1)^j} \to Π^{(m-n)(i+j)} \Ber_{(i+j)/2}^{(-1)^{i+j}},
\end{equation*}
which is a lift of the action of $Γ_\bullet\rtimes \SWitt$ on $\Gr_\bullet$ of \Cref{semidirect action on Gr}.
\end{theorem}
\begin{proof}

Let $S$ be any superscheme. 
By \cref{Ber preserved} and its proof, there exists a sheaf isomorphism $ Π^{(m-n)j} p_\Gr^* \Ber_{j/2}^{(-1)^j} \mathop{\otimes} p_S^* \mc{M} \xrightarrow{\sim}  Π^{(m-n)(i+j)} χ^* p_\Gr^*\Ber_{(i+j)/2}^{(-1)^{i+j}}$ which lifts the action of the element $χ \in (Γ_{i/2} \rtimes \SWitt) (S)$ on $\Gr_{j/2}(m|n)\times S$.

First, notice that we may decompose the semi-direct product group into shift operators and operators preserving the connected components as $Γ_\bullet \rtimes \SWitt\cong \ZZ^2\ltimes \left(Γ_{0/2}^0 \rtimes \SWitt\right)$, where the factor $\ZZ^2$ is generated by the shift operators $Z_{i/2}^A$ defined in \cref{shifts primed ZA}. We wish to imitate the construction in \cite[Thm. 7.7.3]{Pressley.Segal.1986.lg} for the determinant line bundle on the classical Sato Grassmannian, that is first construct the central extension of $Γ_{0/2}^0 \rtimes \SWitt$ and then use the shift operators to extend to the whole group using the action in \cref{shift action on Gr}.

Noticing the action of $Γ_\bullet \rtimes \SWitt$ has orbit of size $\ZZ^2$ within the $\ZZ^3$ components of the Grassmannian, and since we wish to consider the action over any connected component, we label the orbits by $δ\in \ZZ$.
We choose the representatives $\Gr_{0/2}(δ|0)$ of each orbit.

Consider the action of $ξ\in (Γ_{0/2}^0 \rtimes \SWitt)(S)$ on $\Gr_{0/2}(δ|0)\times S$. Let $\tilde{Γ_{0/2}^0 \rtimes \SWitt}(S)$ be the set of pairs $(\tilde{ξ},ξ)$, where $\tilde{ξ}=\left\{\tilde{ξ}^δ \relmiddle| δ\in \ZZ\right\}$ are isomorphisms of line bundles which lift the action of $ξ\in (Γ_{0/2}^0 \rtimes \SWitt)(S)$ on the components $\Gr_{0/2}(\bullet|0)\times S$:
\begin{equation*}
    \begin{tikzcd}
        p_\Gr^*\Ber_{0/2} \mathop{\otimes} p_S^* \mc{M}  \arrow{r}{\tilde{ξ}^δ}\arrow{d}
        &  p_\Gr^*  \Ber_{0/2}\arrow{d}\\
        \Gr_{0/2}(δ|0)\times S \arrow{r}{ξ}
        & \Gr_{0/2}(δ|0)\times S.
    \end{tikzcd}
\end{equation*}
We may compose these pairs by composing the $S$-points $ξ$ of $Γ_{0/2}^0 \rtimes \SWitt$ and their lifts $\tilde{ξ}$ after tensoring them by the pullbacks of suitable line bundles over $S$.
Thus, we get an extension of functors of groups
\begin{equation*}
    1 \to H^0\left( \Gr_{0/2}(\bullet|0)\times S, \mathcal{O}_{\Gr_{0/2}(\bullet|0) \times S}^*\right)
    \to \tilde{Γ_{0/2}^0 \rtimes \SWitt}(S)
    \to (Γ_{0/2}^0 \rtimes \SWitt)(S) \to 1.
\end{equation*}

We claim that $H^0\left(\Gr_{j/2}(m|n)\times S, \mathcal{O}_{\Gr_{j/2}(m|n)\times S}^*\right)=
H^0(S,\mathcal{O}_S^*)$ for each $j, m, n\in \ZZ$. 
The claim applied to $\Gr_{0/2}(δ|0)$ implies the functor of groups $\tilde{Γ_{0/2}^0 \rtimes \SWitt}(S)$ of pairs $(\tilde{ξ},ξ)$ is a central extension
\begin{align*}
    1 \to
    \left(\mathbb{C}^*\right)^{\ZZ} \to  \tilde{Γ_{0/2}^0 \rtimes \SWitt}
    \to Γ_{0/2}^0 \rtimes \SWitt \to 1,
\end{align*}
where by $\mathbb{C}^*$ we mean the multiplicative group $\mathbb{G}_m$ over $\CC$, as $\mathbb{G}_m (S) = H^0(S,\mathcal{O}_S^*)$.
The claim is easy to prove by approximating the super Sato Grassmannian by finite-dimensional super Grassmannians (see \cite[Remark 8]{AlvarezVazquez.MunozPorras.PlazaMartin.1998.tafoseoabf} in the bosonic case), whose global functions are constant by \cite[Proposition 1.1]{Penkov.Skornyakov.1985.paDaofs}.

Since $\tilde{Γ_\bullet \rtimes \SWitt}\cong \ZZ^2\ltimes \tilde{Γ_{0/2}^0 \rtimes \SWitt}$, for $(Z^A_i,χ)\in \ZZ^2\ltimes \left(Γ_{0/2}^0 \rtimes \SWitt\right)
$, we define it's lift by $(Z^A_i,\tilde{χ})\in \ZZ^2\ltimes \tilde{Γ_{0/2}^0 \rtimes \SWitt}$. Thus we have constructed the central extension of the full semi-direct product group, and the corresponding lifts acting:
\begin{equation}
    \label{liftedaction}
    \begin{tikzcd}
        Π^{(m-n)j} p_\Gr^*\Ber_{j/2}^{(-1)^j} \mathop{\otimes} p_S^* \mc{M}  \arrow{rr}{\tilde{\left(Z^A_{i/2},χ\right)}^{m|n}_{j/2}}\arrow{d}
        && Π^{(m-n)(i+j)}  p_\Gr^*  \Ber_{(i+j)/2}^{(-1)^{i+j}}\arrow{d}\\
        \Gr_{j/2}(m|n)\times S \arrow{rr}{\left(Z^A_{i/2},χ\right)}
        && \Gr_{(i+j)/2}(m'|n')\times S,
    \end{tikzcd}
\end{equation}
where $m'|n'$ are as in \cref{shifts primed ZA} and the lift is identified as
\begin{align*}
	\tilde{\left(Z^A_{i/2},χ\right)}^{m|n}_{j/2}\coeq
    \begin{cases}
        Z^{-n}_{-j/2}\left(Z^A_{i/2},\tilde{\left(Z^{n}_{j/2}\,χ\,Z^{-n}_{-j/2}\right)}^{m-n}\right)Z^{n}_{j/2} & j \text{ even},\\[0.6em]
        Z^{-m}_{-j/2}\left(Z^A_{i/2},\tilde{\left(Z^{m}_{j/2}\,χ\,Z^{-m}_{-j/2}\right)}^{n-m}\right)Z^{m}_{j/2} & j\text{ odd}.
    \end{cases}
\end{align*}
\end{proof}

\section{Cocycle Computations}

In this section, we first find an explicit expression for a lift $\tilde{χ}$ of the action of $Γ_\bullet \rtimes \SWitt$ on the Berezinian line bundle, as in \eqref{liftedaction}. This lift allows the associated group 2-cocycles to be defined and its properties studied. Lastly, using the group cocycle properties, we prove two results about the action on the Berezinian line bundle.

\subsection{Expression of the lift to Berezinian line bundles}

We will also work with
$S$-points in this section for $S$ being a superscheme, but not
explicitly mention it for the sake of clarity.

An operator $χ \in Γ^0_{i/2} \rtimes \SWitt$ has a natural decomposition
\begin{align*}
χ\co  H^-_{j/2} \oplus H^+_{j/2} \to H^-_{(i+j)/2} \oplus H^+_{(i+j)/2}.
\end{align*}
Write the corresponding block forms of $χ$ and its inverse as
\begin{align}\label{block decomposition}
χ = 
\begin{pmatrix}
χ^{--} & χ^{-+}\\
χ^{+-} & χ^{++}
\end{pmatrix}, &&
χ^{-1} = 
\begin{pmatrix}
(χ^{-1})^{--} & (χ^{-1})^{-+}\\
(χ^{-1})^{+-} & (χ^{-1})^{++}
\end{pmatrix}.
\end{align}
Consider an element $χ$ sufficiently close to the shift operator $Z_{i/2}^0$ so that $χ^{--}$, $χ^{++}$, $(χ^{-1})^{--}$, and  $(χ^{-1})^{++}$ are invertible.

Given a plane $L \in \Gr_{j/2}(m|n)$, the natural lift of $χ \in Γ^0_{i/2} \rtimes \SWitt$ is defined via the isomorphism of complexes  $\Ber Π^{j}χ^{-1}Π^{i+j}$ as shown in \cref{Ber chi ses}.
This yields an isomorphism
  \begin{multline*}
      \Ber^{-1}\left(Π^{j}χ^{-1}Π^{i+j}\right)\co\\
                Π^{(m-n)j} \Ber_{j/2}^{(-1)^j} (L)
                \otimes \Ber^{(-1)^j} \left( χ^{-1} H^+_{(i+j)/2} \middle/ z^K H^+_{j/2} \right)
                \otimes \Ber^{(-1)^{j+1}} \left( H^+_{j/2} \middle/ z^K H^+_{j/2} \right)\\
                \xrightarrow{\sim} Π^{(m-n)(i+j)} \Ber_{j/2}^{(-1)^{i+j}}(χL),
    \end{multline*}
    where $K\geq 0$ is assumed to be large enough so that $χ z^K H_{j/2}^+ \subset H_{(i+j)/2}^+$. The tensor factors define a line bundle $\mathcal{M}$ over $S$ as described in \cref{Ber preserved}. It remains to cancel these tensor factors, and we do so by choosing a canonical section.

Consider the projection $p_+ \co χ^{-1} H_{(i+j)/2}^+ \to H_{j/2}^+ $ along $H_{j/2}^-$, which is an isomorphism due to the assumption that
$\left(χ^{-1}\right)^{++} = p_+ \circ 
χ^{-1}|_{H_{(i+j)/2}^+}$ is an isomorphism. By decomposing the the source and target super vector spaces via the subspace $z^KH^+_{j/2}$, we may use this isomorphism to induce 
an isomorphism of the Berezinians of the finite dimensional super vector spaces:
\begin{align*}
\Ber p_+\co
   \Ber\left(χ^{-1} H_{(i+j)/2}^+ \Big/ z^K H_{j/2}^+\right)
   \xrightarrow{\sim}
   \Ber\left(H_{j/2}^+ \Big/ z^K H_{j/2}^+\right).
\end{align*}
The resulting canonical section is independent of $K$.

Combining the earlier Berezinian of the morphism $χ^{-1}$ with the section defined by $p_+$ yields an isomorphism between fibers
\begin{multline} \label{lift defintion}
     \tilde{χ}\coeq 
     \Ber^{-1}(Π^{j}χ^{-1}Π^{i+j}) \otimes \Ber ^{(-1)^{j+1}} (p_+)  \co \\
     Π^{(m-n)j}\Ber_{j/2}^{(-1)^j}(L) \to Π^{(m-n)(i+j)}\Ber_{j/2}^{(-1)^{i+j}}(χ L),
\end{multline}
which is equivalently an isomorphism of line bundles
\begin{align*}
    \tilde{χ} \co Π^{(m-n)j}\Ber_{j/2}^{(-1)^j} \to χ^* Π^{(m-n)(i+j)}\Ber_{j/2}^{(-1)^{i+j}}.
\end{align*}

Working properly over a base superscheme $S$, this lift will define a diagram \eqref{liftedaction} with $p_S^* \MM$ trivialized by the section $\Ber^{-1}\left(p_+\right) $. This diagram is exactly what is needed to define an action of the central extension $\tilde{Γ_{\bullet} \rtimes \SWitt}$ on the Berezinian line bundle, as in \Cref{canonical central extension}.

\subsection{The associated group cocycle}

Consider $ξ\in Γ_{i/2}\rtimes \SWitt$ and $χ\in Γ_{k/2}\rtimes \SWitt$. The associated 2-cocycle to the lift  \cref{lift defintion} is then identified as
\begin{align}\label{cocycle basic def}
C^{m|n}_{j/2}(χ, ξ) &\coeq \left(\tilde{χξ}\right)^{-1} \tilde{χ}\tilde{ξ}
\end{align}
where the indices on the cocycle indicate that the composition of lifts acts on some $L\in \Gr_{j/2}(m|n)$ as shown in diagram \cref{liftedaction}, so we may omit indices on the individual lifts without ambiguity.

Unlike the previous section where we restricted to $χ \in Γ^0_{i/2} \rtimes \SWitt$, we consider lifts of arbitrary elements in $Γ_\bullet \rtimes \SWitt$ by utilizing \cref{liftedaction} to extend the lift expression. So while not explicitly written in the below, the appropriately shifted block decomposition of \cref{block decomposition} is meant.

The expression of the lift in \cref{lift defintion} may alternatively be expressed by writing its action on each component of $Π^jL\oplus Π^j H^+_{j/2}$, that is as 
\begin{align*}
    \tilde{ξ}=\Ber\left(Π^{i+j}ξΠ^{j}\oplus Π^{i+j}\left((ξ^{-1})^{++}\right)^{-1}Π^{j}\right).
\end{align*}
Using this notation to identify the composition of these morphisms of perfect complexes, we find:
\begin{align*}
C^{m|n}_{j/2}(χ, ξ) & = \begin{multlined}[t]
    \Ber^{-1}\left(Π^{i+j+k}χξΠ^{j}\oplus Π^{i+j+k}\left(((χξ)^{-1})^{++}\right)^{-1}Π^{j}\right)\\
    \circ \Ber\left(Π^{i+j+k}χΠ^{i+j}\oplus Π^{i+j+k}\left((χ^{-1})^{++}\right)^{-1}Π^{i+j}\right)\\
    \circ \Ber\left(Π^{i+j}ξΠ^{j}\oplus Π^{i+j}\left((ξ^{-1})^{++}\right)^{-1}Π^{j}\right)
\end{multlined}\\
 & = \begin{multlined}[t]
    \Ber \left(Π^{j} \left(χξ\right)^{-1} χ ξ Π^{j}
     \oplus
     Π^{j} \left((χξ)^{-1}\right)^{++} \left((χ^{-1})^{++}\right)^{-1} \left((ξ^{-1})^{++}\right)^{-1} Π^{j} \right)
 \end{multlined}
\end{align*}
Since the first factor inside the Berezinian is in fact the identity on $L$, we can see any contribution to the cocycle is due to the second factor only. If we attempt to calculate the second factor as a Berezinian of an infinite matrix in $\Aut\left(Π^{j}H_{j/2}^+\right)$, the result is in fact finite:
\begin{align}
C^{m|n}_{j/2}(χ, ξ)  &= \nonumber
\Ber \left(  Π^{j} \left((χξ)^{-1}\right)^{++} \left((χ^{-1})^{++}\right)^{-1} \left((ξ^{-1})^{++}\right)^{-1} Π^{j} \right)
\\\nonumber
& = \Ber^{(-1)^{j}} \left( \left((χξ)^{-1}\right)^{++} \left((χ^{-1})^{++}\right)^{-1} \left((ξ^{-1})^{++}\right)^{-1}  \right)\\\nonumber
& = \Ber^{(-1)^{j}} \left( \left(ξ^{--}\right)^{-1}   \left(χ^{--}\right)^{-1} (χξ)^{--} \right)\\
&=
\Ber^{(-1)^{j}}\left(I^{--} +  \left(ξ^{--}\right)^{-1} \left(χ^{--}\right)^{-1}  χ^{-+} ξ^{+-} \right)
\label{group cocycle formula}
\end{align}
This last expression is finite since this matrix differs from the identity on $H_{j/2}^-$ by an automorphism of $H_{j/2}^-$ which factors through $H_{j/2}^+$.

\color{black}

\subsection{Calculation of the algebraic cocycle}

In \cref{semidirect Lie algebra}, we showed that the Lie algebra of $Γ_\bullet\rtimes \SWitt$ is $\mathfrak{h}\rtimes \switt$. 
From the the group-theoretic 2-cocycle $C^{m|n}_{j/2}$, the Lie-algebra 2-cocycle $c^{m|n}_{j/2}$ of $\mathfrak{h}\rtimes \switt$ may be derived.

Let us identify tangent vectors to a complex supermanifold $M$ via \emph{super dual numbers}: a \emph{tangent vector} at a point $p \in M$ is a based at $p$ curve $\Spec \CC[ε_0| ε_1]/(ε_0^2,ε_0ε_1) \to M$, where $\abs{ε_i} = i$. A curve that factors through a curve $\Spec \CC[ε_0]/(ε_0^2) \to M$ is an \emph{even tangent vector}. A curve that factors through a curve $\Spec \CC[ε_1]\to M$ is an \emph{odd tangent vector}.

For even elements $X_0, Y_0$ and odd elements $X_1, Y_1$ of the Lie algebra $\mathfrak{h}\rtimes\switt$ and even and odd dual numbers $ε_0, δ_0$ and $ε_1 , δ_1$, respectively, consider the infinitesimal automorphisms $χ = I + ε_0 X_0 + ε_1 X_1$ and $ξ = I + δ_0 Y_0 + δ_1 Y_1$. For shorthand, we use $εX=ε_0 X_0 + ε_1 X_1$ and $δY=δ_0 Y_0 + δ_1 Y_1$ when confusion is unlikely. Then
\begin{align*}
C_{j/2}^{m|n}(χ,ξ)  = C^{m|n}_{j/2}(I + εX, I + δY)
    = \Ber\left(I^{--} + (-1)^{j} ε X^{-+} δ Y^{+-}  \right),
\end{align*}
and the Lie algebra cocycle can be identified, see \cite[Section 3]{Tuynman.Wiegerinck.1987.ceap}, via 
\begin{align}\label{group algebra cocycle}
    c_{j/2}^{m|n}(X,Y) & =
        \left( \pp{ε_0} \pp{δ_0} + \pp{ε_1} \pp{δ_1} \right) \left(C_{j/2}^{m|n}(χ,ξ) - C^{m|n}_{j/2}(ξ, χ)\right)\\
    &=
     \left( \pp{ε_0} \pp{δ_0} + \pp{ε_1} \pp{δ_1} \right)  (-1)^{j}\str\left( ε X^{-+}δY^{+-}  -δY^{-+}εX^{+-} \right).\nonumber
\end{align}
Decomposing with respect to the super grading
\begin{multline*}
    \str\left(ε X^{-+}δY^{+-}  -δY^{-+}εX^{+-} \right) \\
    \begin{aligned}
        &= ε_0δ_0\str(X_0^{-+}Y_0^{+-}-Y_0^{-+}X_0^{+-}) - ε_1 δ_1\str(X_1^{-+}Y_1^{+-}+Y_1^{-+}X_1^{+-}) \\
        &= ε_0δ_0c^{m|n}_{j/2}(X_0,Y_0)-ε_1δ_1c^{m|n}_{j/2}(X_1,Y_1)
    \end{aligned}
\end{multline*}
gives the Lie algebra cocycle as
\begin{align*}
c^{m|n}_{j/2}(X,Y)&=(-1)^{j}\str\left(X^{-+}Y^{+-}-(-1)^{|X||Y|}Y^{-+}X^{+-}\right).
\end{align*}

In order to calculate the Lie algebraic cocycle, we need to work with explicit generators.
Elements of $\mathfrak{h}$ are simply elements of $H_{0/2}$. By \cref{switt Lie algebra}, elements of $\switt$ may be identified with $[h D_ζ, D_ζ]$ for $h\in\mathbb{C}(\!(z)\!)[ζ]$. The even and odd generators are then
\begin{align*}
			J_p&\coeq  -z^p &	 
			L_q&\coeq [-\tfrac{1}{2}z^{q+1}D_ζ,D_ζ]
        & p,q&\in\mathbb{Z},
			\\ 
			E_r&\coeq   -ζ z^{r-\frac{1}{2}}   &
			G_s&\coeq \left[-\tfrac{1}{2}ζ z^{s+\frac{1}{2}}D_ζ,D_ζ\right]
   & r,s&\in\mathbb{Z}+\tfrac{1}{2}.
\end{align*}
Interpreting such tangent vectors as $\Spec \CC[ε_0|ε_1]$-points of $Γ_\bullet\rtimes \SWitt$ acting on $H_{j/2}$ as linear maps, we may write their matrix entries in the standard basis
\begin{align*}
\left\{z^{a}\; [dz | dζ]^{\otimes j} \relmiddle| z^{b-\frac{1}{2}} ζ \;[dz | dζ]^{\otimes j}\right\} && j \text{ even},\quad a\in \mathbb{Z}, \quad b\in\mathbb{Z}+\tfrac{1}{2},\\
\left\{ z^{b-\frac{1}{2}} ζ \;[dz | dζ]^{\otimes j}  \relmiddle|  z^{a} \;[dz | dζ]^{\otimes j} \right\} && j \text{ odd}, \quad b\in \mathbb{Z}+\tfrac{1}{2}, \quad a\in\mathbb{Z}.
\end{align*}
Writing the tangent vectors as $\mathbb{Z}\times \mathbb{Z}$ supermatrices gives the following matrix entries.
\begin{align*}
            J_p&=\begin{cases}
                (J_p)_{a',a} = ε_0\;δ_{a',a+p}(-1)
   \\
            (J_p)_{b',b} = ε_0\;δ_{b',b+p} (-1)
            \end{cases}
			\\
            E_r&=\begin{cases}
                (E_r)_{a',b} = ε_1\;δ_{a',b+r} (0)
   \\
            (E_r)_{b',a} = ε_1\;δ_{b',a+r} (-1)
            \end{cases}
   \\
            L_q&=\begin{cases}
                (L_p)_{a',a} = ε_0\;δ_{a',a+r} (-(a+\tfrac{1}{2}(q+1)j))
   \\
            (L_q)_{b',b} = ε_0\;δ_{b',b+r}(-(b-\tfrac{1}{2}+\tfrac{1}{2}(q+1)(j+1)))
            \end{cases}
			\\
            G_s&=\begin{cases}
                (G_s)_{a',b} = ε_1\;δ_{a',b+r} (1)
   \\
            (G_s)_{b',a} = ε_1\;δ_{b',a+r}(-(a+\tfrac{1}{2}(2s+1)j))
            \end{cases}
		\end{align*}
For $j$ even, the calculation $\str\left(X^{-+}Y^{+-}\right)$ using the $m|n$ block decomposition as in \cref{block decomposition} is below.
\begin{align*}
        \str(J_p^{-+}J_q^{+-})&
            =\delta_{p+q,0}\left(\sum_{i=m}^{m+q-1} (-1)(-1)
			-\sum_{i=n+\frac{1}{2}}^{n+q-\frac{1}{2}}(-1)(-1)\right)
                \\
            &=0
                \\
        \str(E_r^{-+}E_s^{+-})&
            =\delta_{r+s,0}\left(\sum_{i=n+\frac{1}{2}}^{m+s-1} (-1)(0)
			-\sum_{i=m}^{n+s-\frac{1}{2}}(0)(-1)\right)
                \\
            &=0
            \allowdisplaybreaks\\
        \str(L_p^{-+}L_q^{+-})&
            =\delta_{p+q,0}
            \begin{multlined}[t]
                \left(\sum_{i=m}^{m+q-1} \left(-\left(i+\tfrac{(p+1)j}{2}\right)\right)(-1)\right.\\
                -\left.\sum_{i=n+\frac{1}{2}}^{n+q-\frac{1}{2}}\left(-\left(i-\tfrac{1}{2}+\tfrac{(p+1)(j+1)}{2}\right)\right)(-1)\right)
            \end{multlined}
                \\
            &=\delta_{p+q,0}\;p\left(\tfrac{1}{2}(p+1)-(m-n)\right)
                \\
        \str(G_r^{-+}E_s^{+-})&
            =\delta_{r+s,0}\left(\sum_{i=n+\frac{1}{2}}^{m+s-1}(1)(-1)
			-\sum_{i=m}^{n+s-\frac{1}{2}}\left(-\left(i+\tfrac{1}{2}(2r+1)j\right)\right)(0)\right)
                \\
            &=\delta_{r+s,0}\left(r+\tfrac{1}{2}-(m-n)\right)
                \\
            \allowdisplaybreaks\\
       \str(L_p^{-+}L_q^{+-})&
            =\delta_{p+q,0}
            \begin{multlined}[t]
                \left(\sum_{i=m}^{m+q-1} \left(-\left(i+\tfrac{(p+1)j}{2}\right)\right)\left(-\left(i-q+\tfrac{(q+1)j}{2}\right)\right)\right.\\
                \left.-\sum_{i=n+\frac{1}{2}}^{n+q-\frac{1}{2}}\left(-\left(i-\tfrac{1}{2}+\tfrac{(p+1)(j+1)}{2}\right)\right)\left(-\left(i-q-\tfrac{1}{2}+\tfrac{(q+1)(j+1)}{2}\right)\right)\right)
            \end{multlined}
                \\
			&= \delta_{p+q,0}\;\tfrac{1}{4}p\left( p^2(1-2j) -(1-2(m-n))(1-2j-2(m+n))\right)
                \\
        \str(G_r^{-+}G_s^{+-})&
            =\delta_{r+s,0}
            \begin{multlined}[t]
                \left(\sum_{i=n+\frac{1}{2}}^{m+s-1}(1)\left(-\left(i-s+\tfrac{1}{2}(2s+1)j\right)\right) \right.\\
                -\left. \sum_{i=m}^{n+s-\frac{1}{2}} \left(-\left(i+\tfrac{1}{2}(2r+1)j\right)\right)(1)\right)
            \end{multlined}
                \\
			&=\delta_{r+s,0}\;\left(r^2(1-2j)-\tfrac{1}{4}(1-2(m-n))(1-2j-2(m+n))\right) 
		\end{align*}
  For $j$ odd, the computation of $\str\left(X^{-+}Y^{+-}\right)$ may be derived from the above by swapping $m$ and $n$, and introducing a factor of $-1$ out front. One can check that the cocycle is simply given by $2(-1)^j\str\left(X^{-+}Y^{+-}\right)$.
  
We summarize by stating the commutation relations of the central extension Lie algebra of $\mathfrak{h}\rtimes \switt$:
\begin{align}\label{heisenberg Lie algebra commutation relations}
		[J_p,J_q]&=0
            \\\nonumber
        [J_p,E_s]&=0 
            \\\nonumber
        [E_r,E_s]&=0
            \allowdisplaybreaks
            \\
            \label{mix Lie algebra commutation relations}
		  [L_p,J_q]&=-qJ_{p+q} &+&\delta_{p+q,0}\;p\left(p+1-2δ\right)
            \\\nonumber
        [L_p,E_s]&=-\left(s-\tfrac{1}{2}+\tfrac{1}{2}(p+1)\right)E_{p+s} \hspace{-3em}
            \\\nonumber
        [G_r,J_q]&=-qE_{q+r}
            \\\nonumber
        [G_r,E_s]&=J_{r+s}  &+&\delta_{r+s,0}\left(2r+1-2δ\right)
            \allowdisplaybreaks
            \\
            \label{ns Lie algebra commutation relations}
		[L_p,L_q]&=(p-q)L_{p+q} &+&\delta_{p+q,0}\; \tfrac{1}{2}p\left( p^2(1-2j) -(1-2δ)(1-2j-2σ)\right)
            \\\nonumber
		[L_p,G_s]&=\left(\tfrac{p}{2}-s\right)G_{p+s} 
            \\\nonumber
        [G_r,G_s]&=2L_{r+s} &+&\delta_{r+s,0}\left(2r^2(1-2j)-\tfrac{1}{2}(1-2δ)(1-2j-2σ)\right)
	\end{align}
 where we let $δ=m-n$ for $j$ even and $δ=n-m$ for $j$ odd, and $σ=m+n$.

\subsection{Properties of the group cocycle}

\begin{theorem}
\label{Gamma-action}
The action $Γ_{\bullet} \times \Gr_{\bullet} \to \Gr_{\bullet}$ in \cref{gamma action on Gr} lifts to an action on the Berezinian line bundle.
\end{theorem}
\begin{proof}
We show this result by first showing that the group cocycle $C^{m|n}_{j/2}$ restricted to the group $Γ_{\bullet}$ is trivial.

As stated in \cref{heisenberg Lie algebra commutation relations}, the Lie algebraic cocycle $c^{m|n}_{j/2}$ restricted to the Lie algebra $\mathfrak{h}$ is trivial:
    \begin{align*}
        c^{m|n}_{j/2}(F,G)=0, && F,G\in \mathfrak{h}.
    \end{align*}
    From this fact and the relationship between the group and Lie algebra cocycles stated in \cref{group algebra cocycle}, we claim that $C^{m|n}_{j/2}(f,g)=1$ for $f,g\in Γ^0_{0/2}$. This can be seen by first noticing that in fact the Lie algebra cocycle is proportional to the first order approximation to a single group cocycle not just the difference of cocycles, and then noticing that $\str(X^{-+}Y^{+-})=0$ implies all the higher terms vanish in the expansion of the group cocycle beyond order $1$ in $εδ$.

    We showed in \cref{shift action on Gr} that the shift operators $Z^A_{i/2}$ act on the Berezinian line bundle directly. This direct action is reflected in the group cocycle formula \cref{group cocycle formula} in that the block components of the shift operators $\left(Z^A_{i/2}\right)^{-+}=\left(Z^A_{i/2}\right)^{-+}=0$, which implies that $C^{m|n}_{j/2}\left(χ,Z^A_{i/2}\right)=C^{m|n}_{j/2}\left(Z^A_{i/2},χ\right)=1$ for any $χ\in Γ_\bullet\rtimes \SWitt$.

    Therefore it is only left to show that $C^{m|n}_{j/2}\left(f Z^A_{i/2}, g Z^{B}_{k/2}\right)=1$ for any $f,g\in Γ^0_{0/2}$ and $A,B,i,k\in\ZZ$. This follows from two applications of the group cocycle condition:
    \begin{align*}
        C\left(f Z^A_{i/2}, g Z^{B}_{k/2}\right) &= C\left(Z^A_{i/2}, g Z^{B}_{k/2}\right) \left(C\left(f, Z^A_{i/2}\right)\right)^{-1} C\left(f , g Z^{A+B}_{(i+k)/2}\right)\\
        & =         C\left(f , g Z^{A+B}_{(i+k)/2}\right)\\
        & =  C\left(f  g, Z^{A+B}_{(i+k)/2}\right) \left( C\left( g ,Z^{A+B}_{(i+k)/2}\right)\right)^{-1}  C\left(f , g\right)\\
        & = 1.
    \end{align*}
\end{proof}

\begin{prop}\label{mixed cocycle independent of j}
    For $g\in Γ_{\bullet}$ and $φ\in\SWitt$, the group cocycle $C^{m|n}_{j/2}$ of \cref{canonical central extension} with explicit expression in \cref{cocycle basic def} has the property that
    \begin{align*}
        C^{m|n}_{j/2}(φ,g)&=C^δ(φ,g) ,
        & 
        C^{m|n}_{j/2}(g,φ)&=C^δ(g,φ) ,
    \end{align*}
    for $δ=m-n$ for $j$ even, and $δ=n-m$ for $j$ odd, and where $C^δ \coeq C^{δ|0}_{0/2}$.
\end{prop}
\begin{proof}
    As stated in \cref{mix Lie algebra commutation relations}, the Lie algebraic cocycle $c^{m|n}_{j/2}$ has the property
    \begin{align*}
        c^{m|n}_{j/2}(X,F)&=c^{δ}(X,F) ,
        &
        c^{m|n}_{j/2}(F,X)&=c^{δ}(F,X) ,
    \end{align*}
    for $X\in\switt$ and $F\in \mathfrak{h}$, and $δ=m-n$ for $j$ even and $δ=n-m$ for $j$ odd, where $c^{δ}=c^{δ|0}_{0/2}$.
    Using the relationship between the group and Lie algebra cocycles as in \cref{group algebra cocycle}, we claim it follows that
     \begin{align*}
        C^{m|n}_{j/2}(φ,f)&=C^{δ}(φ,f) ,
        &
        C^{m|n}_{j/2}(f,φ)&=C^{δ}(f,φ) ,
    \end{align*}
    for $f\in Γ^0_{0/2}$ and $φ\in\SWitt$. This can be seen by a similar argument as in \cref{Gamma-action} using a series expansion of the group cocycles.

    It only remains to show the claim with an arbitrary element of the super Heisenberg group $f Z^{A}_{i/2}\in Γ_{\bullet}$ with $f\in Γ^0_{0/2}$. This follows from an application of the group cocycle condition:
\begin{align*}
	 C\left(φ , f Z^{A}_{i/2}\right)     & =  C\left(φ  f, Z^{A}_{i/2}\right) \left( C\left( f ,Z^{A}_{i/2}\right)\right)^{-1}  C\left(φ, f\right)\\
        & = C\left(φ, f\right),
\end{align*}
and similarly for the other ordering of arguments.
\end{proof}

\section{The Neveu-Schwarz group}

Noticing that the action of super Witt group preserves the connected component of the Grassmannian, see \cref{SWitt action on Gr}, then there should be a 1 dimensional central extension associated to the action on the Berezinian line bundle on each connected component of the Grassmannian. Since the action on the Berezinian line bundle has been analyzed in \cref{action on Ber section} for the semi-direct product group $Γ_\bullet \rtimes \SWitt$, we can easily derive the Neveu-Schwarz groups using $\tilde{Γ_\bullet \rtimes \SWitt}$.

\begin{definition}[The Neveu-Schwarz group $\NS$]
\label{NS}
Define \emph{Neveu-Schwarz formal group} $\NS$ as the pullback of $\tilde{Γ_\bullet \rtimes \SWitt}$ (defined in \cref{canonical central extension}) along the inclusion $\SWitt\hookrightarrow Γ_{\bullet}\rtimes \SWitt$:
\begin{equation*}
    \begin{tikzcd}
        1\arrow{r}&
        \left(\mathbb{C}^*\right)^{\ZZ^3} \arrow[two heads]{d}{δ}\arrow{r}&  \NS
        \arrow{r} \arrow{d} \arrow[very near start, phantom]{dr}{\lrcorner} & \SWitt \arrow{r} \arrow[hook]{d} &1\\
         1\arrow{r} &
         \left(\mathbb{C}^*\right)^{\ZZ} \arrow{r} &  \tilde{Γ_{\bullet} \rtimes \SWitt}
        \arrow{r} & Γ_{\bullet}\rtimes \SWitt \arrow{r} & 1
    \end{tikzcd}
\end{equation*}
where the induced morphism of the centers is the projection $\left(\mathbb{\CC}^*\right)^{\ZZ^3} \twoheadrightarrow \left(\mathbb{\CC}^*\right)^{\ZZ}$ from each connected component of the Grassmannian onto $\Gr_{0/2}(δ|0)$.
In particular, an element $\tilde{φ}\in \NS$ is the collection of $\ZZ^3$ isomorphisms of line bundles $\tilde{φ}^{m|n}_{j/2}$ (one for each component of $\Gr_{j/2}^{m|n}$) which map to $\tilde{\left(0,φ\right)}^{m|n}_{j/2}\in \tilde{Γ_{\bullet} \rtimes \SWitt}$ defined in \cref{liftedaction}.
\end{definition}
Note that we may define 1-dimensional central extensions of $\SWitt$ denoted as $\NS^{m|n}_{j/2}$ by the group of lifts $\tilde{φ}^{m|n}_{j/2}$ for fixed indices. These extensions may equivalently be defined by the restriction of the cocycle $C^{m|n}_{j/2}$ to $\SWitt$.

\begin{cor}
    The Neveu-Schwarz algebra $\mathfrak{ns}$ is the Lie algebra of the formal group superscheme $\NS^{0|0}_{0/2}$.
\end{cor}
\begin{proof}
    The proof is simply letting $j=0$, $δ=0$, and $σ=0$ in equation \cref{ns Lie algebra commutation relations} to arrive at the standard cocycle of the Neveu-Schwarz algebra $\mathfrak{ns}$ (at least, up to a factor of $2$ on the center).
    \begin{align*}
        [L_p,L_q]&=(p-q)L_{p+q}+\delta_{p+q,0} \tfrac{p^3-p}{2} \\
        [L_p,G_s]&=\left(p-2s\right)G_{p+s}\\
        [G_r,G_s]&=2L_{r+s}+\delta_{r+s,0}\left(2r^2-\tfrac{1}{2}\right)
    \end{align*}
\end{proof}

\section{Schwarz's super \texorpdfstring{$\tau$}{τ} function}

In this section, we define Schwarz's tau function adapted to our setting of the graded Grassmannian $\Gr_{j/2}(m|n)$.
\Cref{Gamma-action} shows that the super Heisenberg group acts on the Berezinian line bundle without a center, which allows for Schwarz's tau function to be defined for arbitrary elements of $Γ_\bullet$.

\begin{definition}[{\textcite{Schwarz.1989.fsaums}}] \label{schwarz tau def}
    Let $L\in \Gr_{j/2}(m|n)$ and $g\in Γ_{i/2}$. Then \emph{Schwarz's super tau function} is defined as
    \begin{multline*}
    τ_L(g)\coeq  \Ber\Big( Π^{i+j}gL\xrightarrow{Π^{j}g^{-1}Π^{i+j}} Π^jL\Big) \\
    \in H^0\Big(Γ_{i/2}\times \Gr_{j/2},\; Π^{(m-n)i} γ^* \Ber_{(i+j)/2}^{-(-1)^{i+j}} \otimes \Ber_{j/2}^{(-1)^j}\Big),
\end{multline*}
    where the Berezinian denotes the section defined by the action of
    \begin{align*}
        g^{-1}\co Π^{(m-n)(i+j)}\Ber_{H^+}^{(-1)^{i+j}} \to (g^{-1})^* Π^{(m-n)(j)}\Ber_{H^+}^{(-1)^{j}},
    \end{align*}
with $γ\co Γ_{i/2}\times \Gr_{j/2}\to \Gr_{(i+j)/2}$ being the action of the super Heisenberg group on the Grassmannian as in \Cref{gamma action on Gr}.
\end{definition}

We also have
\begin{equation*}
\tau_{gL} (g^{-1}) = \tau_{L}^{-1} (g),
\end{equation*}
which follows from a more general identity
\begin{equation}
\label{taufg}
    \tau_{L} (fg) = \tau_{gL}(f) \, \tau_{L} (g) .
\end{equation}

Define the action of $χ\in G$ on sections $s_L(g)$ of a $G$-equivariant sheaf on $Γ_{i/2}\times \Gr_{j/2}$ as
\begin{equation*}
    (\tilde{χ} s)_L (g) \coeq \tilde{χ} \left(s_{χ^{-1} L}(\Ad_{χ^{-1}} g)\right).
\end{equation*}

\begin{lemma}\label{NS action on tau}
    The action of $\NS$ on Schwarz's super tau function is given by a multiplicative factor. Moreover, for $\tilde{φ}\in\NS$, $g\inΓ_{\bullet}$, and $L\in\Gr_{j/2}(m|n)$, the action is explicitly
\begin{align*}
\left(\tilde{φ} τ\right)_L (g) =   \frac{C^δ \left(φ,\Ad_{φ^{-1}}(g^{-1})\right)}{C^δ\left(g^{-1},φ\right)}  \cdot {τ}_{L}(g),
\end{align*}
for $δ=m-n$ for $j$ even, and $δ=n-m$ for $j$ odd, and where $C^δ \coeq C^{δ|0}_{0/2}$.
\end{lemma}

\begin{proof}
The line bundle $γ^*\Ber^*_{(i+j)/2} \otimes  \Ber_{j/2}$ over $Γ_{i/2}\times \Gr_{j/2}$ acquires the structure of an NS-equivariant line bundle as the tensor product of two NS-equivariant line bundles. The line bundle $γ^*\Ber^*_{(i+j)/2}$ becomes NS-equivariant as the pullback of an NS-equivariant line bundle via an NS-equivariant map $γ\co Γ_{i/2}\times \Gr_{j/2} \to \Gr_{(i+j)/2}$. Thus, an element $\tilde{\phi}$ of the Neveu-Schwarz group acts on a section $γ^* s_1 \otimes s_2$ of $γ^*\Ber^*_{(i+j)/2} \otimes  \Ber_{j/2}$ by the formula
\begin{equation*}
\tilde \phi s_1\left( \phi^{-1}(gL)\right) \otimes \tilde \phi s_2(\phi^{-1} L).
\end{equation*}
Hence, the action on the super tau function will be
\begin{equation*}
    \tilde{\phi} \left(\tau_{\phi^{-1}L} (\Ad_{\phi^{-1}}g)    \right) = \tilde \phi  \Ber\left( Π^{i+j}\phi^{-1} g L \xrightarrow{Π^j\Ad_{\phi^{-1}}(g^{-1})Π^{i+j}}   Π^j\phi^{-1} L\right)
\end{equation*}

Recall the definition of the group cocycle $C^{m|n}_{j/2}$ in \cref{cocycle basic def}. The projective factor which the cocycle represents is equivalent to the cyclic permutations of the lifts within the cocycle definition. So alternatively, we have $C(χ, ξ) \cong  \tilde{χ}\tilde{ξ}\left(\tilde{χξ}\right)^{-1} \cong \tilde{ξ}\left(\tilde{χξ}\right)^{-1}  \tilde{χ}$.

It is then a simple matter of algebraic manipulation to write the adjoint action of $\tilde{φ}\in \NS$ on the lift $\tilde{\Ad_{φ^{-1}}g^{-1}} $ as a multiple of the original lift $\tilde{g^{-1}}$:
\begin{align*}
     \tilde{φ}\left(\tilde{\Ad_{φ^{-1}}g^{-1}} \right)\left(\tilde{φ}\right)^{-1}& = \tilde{g^{-1}}\left(\tilde{g^{-1}}\right)^{-1} \tilde{g^{-1}φ}\left(\left(\tilde{g^{-1}φ}\right)^{-1}\tilde{φ} \left(\tilde{\Ad_{φ^{-1}}g^{-1}}\right)\right)\left(\tilde{φ}\right)^{-1}\\
     &=\tilde{g^{-1}}\left(\left(\tilde{g^{-1}}\right)^{-1} \tilde{g^{-1}φ}\;\left(\tilde{φ}\right)^{-1} \right) C_{(i+j)/2}^{m'|n'} \left(φ,\Ad_{φ^{-1}}(g^{-1})\right)\\
     &=\tilde{g^{-1}}\left(C_{(i+j)/2}^{m'|n'}\left(g^{-1},φ\right)\right)^{-1} C_{(i+j)/2}^{m'|n'} \left(φ,\Ad_{φ^{-1}}(g^{-1})\right).
\end{align*}
Here we have identified the indices on the two cocycles as $m'|n'$ as defined in \cref{shifts primed ZA} for $g\in Γ^A_{i/2}$.
Lastly, we may apply \cref{mixed cocycle independent of j} to identify the two cocycles both as $C^δ$.
\end{proof}

\begin{lemma}\label{Gamma action on tau}
    Schwarz's super tau function $τ_L(g)$ is invariant under the action of the super Heisenberg group $Γ_\bullet$.
\end{lemma}
\begin{proof}
    The line bundle $γ^*\Ber^*_{(i+j)/2} \otimes  \Ber_{j/2}$ is $Γ_\bullet$-invariant for similar reasoning to that in the previous lemma for $\NS$. Further the action on a section by $f\in Γ_\bullet$, is given by:
\begin{align*}
    (f τ)_L(g) & =
    f \left(\tau_{f^{-1}L} (\Ad_{f^{-1}}g)    \right)  = f \left(\tau_{f^{-1}L} (g)    \right)\\
    &= f  \Ber\left( Π^{i+j}f^{-1} g L \xrightarrow{Π^j g^{-1}Π^{i+j}}   Π^j f^{-1} L\right) = τ_L(g).
\end{align*}
Here we have used the commutivity of $Γ_\bullet$ to show $\Ad_{f^{-1}}g = f^{-1} gf = g$, and the action without center of $Γ_\bullet$ on the Berezinian line bundle shown in \cref{Gamma-action}.
\end{proof}

\color{black}
\section{Duality on the super Sato Grassmannian}\label{duality section}

In order to combine Schwarz's super tau functions into an extended super Mumford form, we need to restrict to a certain locus within the super Sato Grassmannian. This locus is characterized by the duality discussed in this section.

Just as in the classical case, for the super vector space $H_\bullet$ there is a natural bilinear scalar product, which in the super case is given by
\begin{align}\label{innerproduct}
   \langle v,w\rangle \coeq\oint_{\mathbb{S}^{1|1}} v\cdot w
\end{align}
In coordinates $(z|ζ)$ this is given by
\begin{align*}
    \Big\langle (v_0+v_1ζ)[dz|dζ]^j,(w_0+w_1ζ)[dz|dζ]^k\Big\rangle & = \delta_{j+k,1}\oint_{\mathbb{S}^{1|1}} (v_0+v_1ζ)(w_0+w_1ζ)[dz|dζ]\\
    &=\delta_{j+k,1}\oint_{\mathbb{S}^{1}} (v_0w_1+v_1w_0)dz.
\end{align*}

\begin{lemma}\label{invariant inner product}
\begin{enumerate}[$(1)$]
\item The inner product \eqref{innerproduct} on the super vector space $H_\bullet$ is $Γ_\bullet$-invariant, if we define the left action of $g \in Γ_\bullet$ on the first factor as multiplication by $g^{-1}$ on the right and on the second factor as multiplication by $g$ on the left:
\begin{equation*}
\langle v g^{-1}, g w \rangle = \langle v, w \rangle.
\end{equation*}
    \item 
The inner product \eqref{innerproduct}  is $\SWitt$-invariant, where $\SWitt$ acts as in \eqref{switt H action}:
\begin{equation*}
\langle \phi v, \phi  w \rangle = \langle v, w \rangle.
\end{equation*}
\end{enumerate}
\end{lemma}

\begin{proof}
Part (1) is obvious from \eqref{innerproduct}. For Part (2) note that the inner product \eqref{innerproduct} is the composition of the multiplication map on $H_\bullet$ and the residue at the origin. Since the multiplication map 
   \begin{align*}
       H_\bullet\otimes H_\bullet \to H_\bullet
   \end{align*}
   is $\SWitt$-equivariant and the residue map, given by the integral
   \begin{align*}
   \oint_{\mathbb{S}^{1|1}} \co H_\bullet\to \mathbb{C},
   \end{align*}
   is invariant with respect to coordinate changes, which is what $\SWitt$ does, the inner product will be  $\SWitt$-invariant.
\end{proof}

\begin{rem}
Let $H_{j/2}^*$ denote the continuous dual space of $H_{j/2}$ with respect to the $z$-adic topology.    Then there is a canonical \emph{duality isomorphism}
    \begin{align*}
         H_{j/2} & \cong (H_{(1-j)/2})^*\\
         v & \mapsto \langle v, - \rangle.
    \end{align*}
\cref{invariant inner product} implies that the duality isomorphism
\begin{align*}
     H_{\bullet/2} \to (H_{(1-\bullet)/2})^*
\end{align*}
is $\SWitt$-equivariant. It will also be $Γ$-equivariant, if we use the left action of $g \in Γ$ on $H_{\bullet/2}$ by multiplication by $g^{-1}$ and the standard left multiplication by $g$ on $H_{(1-\bullet)/2}$.
\end{rem}

We now use the inner product defined above to define the duality on the super Sato Grassmannian. See Section 2.E of \cite{MunozPorras.PlazaMartin.1999.eotmopcitiG} for a description of duality on the classical Sato Grassmannian.

Let $S$ be a $\CC$-superscheme and $\hat{(H_{j/2})}_S$ be the completed trivial vector bundle with fiber $H_{j/2}$ over $S$.

\begin{definition} Consider an $S$-point of the super Sato Grassmannian $L$. That is to say $L\subset \hat{(H_{j/2})}_S$ is discrete. Define the \emph{orthogonal complement of $L$} by
\begin{equation*}
    L^\perp(U) \coeq \Big\{ v\in \hat{(H_{j/2})}_S(U) \mathrel{\Big|}
    \langle v, w \rangle =0 \text{ for all $w \in L(U)$} \Big\}
\end{equation*}
for each open $U \subset S$.
The (\emph{Serre}) \emph{duality map} on the super Sato Grassmannian is given by
\begin{align*}
    \perp\co\Gr_{j/2}(m|n) &\to \Gr_{(1-j)/2}(-m|-n), &
    L &\mapsto L^\perp,
\end{align*}
where we define $\Gr_{j/2}(m|n)$ to be the connected component of the super Sato Grassmannian of discrete subspaces $D$ which have Fredholm index $(m|n)$, i.e., the Fredholm index of the operator $D\oplus H^+_{j/2}\to H_{j/2}$, where $H^+_{j/2}\coeq \CC [\![z]\!][ζ]\;[dz|dζ]^{\otimes j}$ is the distinguished compact subspace.
\end{definition}

\begin{rem}
\label{perp}
    Observe that $\Ber_{(1-j)/2} (L^\perp) = \Ber_{j/2} (L)$ and therefore we have a canonical isomorphism of line bundles over $\Gr_{j/2}$:
    \begin{equation*}
    \perp^* \Ber_{(1-j)/2} = \Ber_{j/2}.
    \end{equation*}
\end{rem}

\section{Schwarz's locus}

In preparation to define Schwarz's extended Mumford form, we define in this section the locus within the Grassmannian which will be used.

\begin{definition} \emph{Schwarz's locus} $\mathfrak{U}_{j/2}$ is defined by the functor of points $\operatorname{\mathbf{SSch}}_\CC \to \operatorname{\mathbf{Set}}$
\begin{align*}
    \mathfrak{U}_{j/2}(m|n) (S)  \coeq \Big\{L\in\Gr_{j/2}(m|n)(S) \mathrel{\Big|} g L = L^\perp\text{ for some } g\inΓ_\bullet(S)\Big\}.
\end{align*}
\end{definition}

\begin{rem}
For $L\in \mathfrak{U}_{j/2}(m|n)$ and $g\in Γ_{\bullet}$ such that $g L=L^\perp$, we have $g\inΓ^{-(m+n)}_{(1-2j)/2}$.
\end{rem}

One can equivalently define Schwarz's locus as follows. Let $\tilde{\mathfrak{U}}_{j/2}$ be defined by the functor of points
\begin{equation*}
S \mapsto \tilde{\mathfrak{U}}_{j/2}(S) \coeq \Big\{(g,L)\in Γ_{(1-2j)/2} (S) \times \Gr_{j/2}(S) \mathrel{\Big|} g L = L^\perp \Big\}.
\end{equation*}
Then 
\begin{equation*}
\mathfrak{U}_{j/2}(S) = p_2 \left(\tilde{\mathfrak{U}}_{j/2}(S)\right),
\end{equation*}
where $p_2\co Γ_{(1-2j)/2} (S) \times \Gr_{j/2}(S) \to \Gr_{j/2}(S)$ is the projection onto the second factor.

\begin{prop}[{\textcite{Schwarz.1989.fsaums}}]
  The $j$th super Krichever map $\kappa_{j/2} \co \mathfrak{M}_{g,1_\NS^\infty} \to \Gr_{j/2}$ maps the moduli space $\mathfrak{M}_{g,1_\NS^\infty}$ to Schwarz's locus $\mathfrak{U}_{j/2}$,
  \begin{equation*}
 \kappa_{j/2} \co \mathfrak{M}_{g,1_\NS^\infty} \to \mathfrak{U}_{j/2}.
  \end{equation*}
\end{prop}

\begin{proof}
It suffices to show that for a small and simple enough (with a trivial Picard group in the \'{e}tale topology or disk-like in the complex topology) neighborhood $S$ of each point of $\mathfrak{M}_{g,1_\NS^\infty}$, the $S$-point of the Grassmannian $\Gr_{j/2}$ obtained by composing with $\kappa_{j/2}$ is in fact an $S$-point of $\mathfrak{U}_{j/2}$.

An $S$-point of the moduli space $\mathfrak{M}_{g,1_\NS^\infty}$ is represented by a family $X \to S$ of SRSs  with an NS puncture, given by a section $P$, and a formal coordinate at it over the superscheme $S$. The image of $S$ under the super Krichever map $\kappa_{j/2}$ is represented by the $S$-family of subbundles $π_* ω_{X/S}^{ j} \subset \hat{(H_{j/2})}_S$, where $π\co X \setminus P \to S$ is the restriction of the family to $X \setminus P$ and $ω_{X/S} \coeq \Ber \Omega_{X/S }^1$ is the \emph{relative dualizing sheaf}.
    \begin{lemma}[{\textcite{Schwarz.1989.fsaums}}]
    \label{explicit-perp}
    $(π_*ω_{X/S}^{ j})^\perp = π_* ω_{X/S}^{1- j}$.
    \end{lemma}
    \begin{proof}[Proof of Lemma]
        The inclusion $π_* ω_{X/S}^{1- j} \subset (π_* ω_{X/S}^{ j})^\perp$ comes from the observation that the product of local sections of these bundles over $S$ is regular on $X \setminus P$ and therefore its residue at $P$ must be zero, since the sum of residues over all points must vanish \cite{Rosly.Schwarz.Voronov.1988.gosm}. To show that the inclusion is actually an equality, for each closed point $t$ of the base $S$ of the family, consider the special fiber $X_t$ over $k(t) = \CC$. This is an individual SRS over $\CC$. We claim that
        \begin{equation}
        \label{equality}
        Γ(X_t \setminus P(t),  ω_{X_t}^{ j})^\perp = Γ(X_t \setminus P(t),  ω_{X_t}^{1- j}).
\end{equation}
        If this is the case, then by Nakayama's lemma, the same will be true in an open neighborhood of $t$, which will imply the lemma.
        
        To prove \cref{equality}, it suffices to show that the Fredholm indices of $Γ(X_t \setminus P(t),  ω_{X_t}^{j})^\perp$ and $Γ(X_t \setminus P(t),  ω_{X_t}^{1-j})$ as closed points of $\Gr_{j/2}$ are equal. Note that the Fredholm index of  $Γ(X_t \setminus p,  ω_{X_t}^{j})^\perp$ is the negative of that of $Γ(X_t \setminus P(t),  ω_{X_t}^{j})$.  The Fredholm index of $Γ(X_t \setminus P(t),  ω_{X_t}^{j})$ is equal to the Euler characteristic of $ω_{X_t}^{j}$, computed  by the super Riemann-Roch theorem \cite{Rosly.Schwarz.Voronov.1988.gosm} as $χ(X_t, ω_{X_t}^{j}) = (d + 1 - g \, | \,  d) = (j(g-1) + 1 -g \, | \, j(g-1)) = ((j-1)(g-1) \, | \, j(g-1))$ if $j$ is even and $χ(X_t, ω_{X_t}^{j}) = (d \, | \,  d + 1 - g) = (j(g-1) \, | \, (j-1)(g-1)) $ if $j$ is odd. On the other hand, the Fredholm index of $Γ(X \setminus P(t),  ω_{X/S}^{1- j})$ is $χ(X_t, ω_{X_t}^{1-j}) = ((1-j)(g-1) \, | \, -j(g-1) )$ or $χ(X_t, ω_{X_t}^{1-j}) = (-j(g-1) \, | \, (1-j)(g-1))$, respectively. This implies \cref{equality} and finishes the proof of the lemma.
\let\qed\relax
\end{proof}

To deduce the statement of the proposition from the lemma, observe that the line bundles $ω_{X/S}^{j}$ and $Π ω_{X/S}^{1-j}$ are isomorphic over $X\setminus P$, provided $S$ is small and simple enough in the \'{e}tale or complex topology, so as $\Pic (S)$ is trivial. Indeed, the relative Picard group $\Pic_{(X \setminus P) /S}(S) = \Pic (X \setminus P)/ π^* \Pic(S)$ is trivial, and therefore, so is the absolute one, $\Pic (X \setminus P)$.

The isomorphism $ω_{X/S}^{j} \xrightarrow{\sim} Π ω_{X/S}^{1-j}$ is given by an invertible regular section $Π g$ of $Π ω_{X/S}^{1-2j}$ over $X \setminus P$. Since $ω_{X/S}^{1-j} = g \cdot ω_{X/S}^{j}$, we also have $π_* ω_{X/S}^{1-j} = g \cdot π_* ω_{X/S}^{j}$. Finally, \cref{explicit-perp} implies $(π_*ω_{X/S}^{ j})^\perp = g \cdot π_* ω_{X/S}^{j}$, which means that the $S$-point $\kappa_{j/2} (X \to S) = π_*ω_{X/S}^{ j}$ is an $S$-point of Schwarz's locus $\mathfrak{U}_{j/2}$, i.e., $S$ maps to $\mathfrak{U}_{j/2}$ by the super Krichever map.
\end{proof}

\begin{prop}
 Schwarz's locus is preserved under the action of
\begin{enumerate}[$(1)$]
\item the super Heisenberg algebra
\item and the super Witt algebra.
\end{enumerate}
\end{prop}
\begin{proof}
\begin{enumerate}[$(1)$]
\item Suppose $g'$ is an element (an $S$-point, to be precise) of $Γ_\bullet$. By definition, for each discrete plane $L_S \subset \hat{(H_{j/2})}_S$ in Schwarz's locus, we have $L_S^\perp = g L_S$ for some $g \in Γ_\bullet(S)$. Then $(g' L_S)^\perp = g' (L_S)^\perp  = g' (g L_S)=g(g'L_S) $ from \cref{invariant inner product} and the commutativity of $Γ_\bullet$. Thus, $g' L_S$ is in Schwarz's locus.
\item Suppose $\phi$ is an element (an $S$-point, to be precise) of $\SWitt$. By definition, for each discrete plane $L_S \subset \hat{(H_{j/2})}_S$ in Schwarz's locus, we have $L_S^\perp = g L_S$ for some $g \in Γ_\bullet(S)$. Then $(\phi L_S)^\perp = \phi (L_S)^\perp  = \phi (g L_S) = \phi(g) \phi (L_S) $ from \cref{invariant inner product,equivariant gamma action}. Thus, $\phi L_S$ is in Schwarz's locus.
\end{enumerate}
\end{proof}

\section{Schwarz's extended super Mumford form}\label{main result section}

Using a simple combination of Sch\-warz's tau functions, then Schwarz's extended super Mumford form is defined.
Recall the action of the super Heisenberg group on the Grassmannian as in \Cref{gamma action on Gr}. This action for Schwarz's locus (in it's cube version) and the duality of \cref{duality section} allow to Schwarz's extended form to be identified as a section $M(L)$ over Schwarz's locus $\mathfrak{U}_{j/2}$.

Out main result is the invariance of the section $M(L)$ under the super Heisenberg group $Γ_\bullet$ and the Neveu-Schwarz group $\NS$.

\begin{prop}
    For $(g,L)$ in $\tilde{\mathfrak{U}}_{j/2}(m|n)$, then Schwarz's tau function with these inputs is a section as follows:
    \begin{equation*}
    \tau_L(g) \in
    \begin{cases}
        H^0\left(\tilde{\mathfrak{U}}_{j/2}, Π^{m-n} p_2^* \Ber^2_{j/2}\right) & j \text{ even},\\
        H^0\left(\tilde{\mathfrak{U}}_{j/2}, Π^{m-n} p_2^* \Ber^{-2}_{j/2}\right) & j \text{ odd}.
    \end{cases}
\end{equation*}
\end{prop}

\begin{proof}
    This follows from \cref{perp} and the definition of Schwarz's tau function.
\end{proof}

\begin{definition}
Let $(g,L) \in \tilde{\mathfrak{U}}_{j/2}$.
Then \emph{Schwarz's extended Mumford form} is defined as
\begin{align}
\label{SchwarzMumfordform}
    M(g,L)\coeq \frac{{\tau}_{L}(g^3)}{{\tau}^3_{L}(g)} = \frac{{\tau}_{gL}(g^2)}{{\tau}^2_{L}(g)}.
\end{align}
This formula defines a global section over Schwarz's locus as
\begin{align*}
    M(g,L) \in \begin{cases}
        H^0\left(\tilde{\mathfrak{U}}_{j/2},γ^* \Ber_{(-5j+3)/2} \otimes \big(\Ber^*_{j/2}\big)^5\right) & j \text{ even},\\[0.8em]
        H^0\left(\tilde{\mathfrak{U}}_{j/2}, γ^* \Ber^*_{(-5j+3)/2} \otimes \big(\Ber_{j/2}\big)^5\right) & j \text{ odd},\\
               \end{cases}
\end{align*}
where $γ$ is the “\emph{cube action}” map
\begin{align*}
    γ\co \tilde{\mathfrak{U}}_{j/2} & \to \Gr_{(-5j+ 3)/2},\\
(g,L) & \mapsto g^3 L.
\end{align*}
\end{definition}


We would like to show that $M(g,L)$ is actually independent of $g$ and thus can be considered as a section of a line bundle over $\mathfrak{U}_{j/2}$.

\begin{prop}
The “cube action” map $γ$ factors through the projection $p_2\co \tilde{\mathfrak{U}}_{j/2} \to \mathfrak{U}_{j/2} $, i.e., there is a unique morphism $\hat{γ}$  which makes the following triangle commute:
\begin{equation*}
 \begin{tikzcd}
\tilde{\mathfrak{U}}_{j/2} \arrow{rr}{γ} \arrow{dr}{p_2} & & \Gr_{(-5j+ 3)/2}.\\
& \mathfrak{U}_{j/2} \arrow{ur}{\hat{γ}}
    \end{tikzcd}
\end{equation*}
\end{prop}
\begin{proof}
    If $\hat{γ}$ exists, it must work the following way on $S$-points of $\mathfrak{U}_{j/2}$:
\begin{align*}
    \hat{γ}\co \mathfrak{U}_{j/2}(S) & \to \Gr_{(-5j+ 3)/2} (S),\\
L & \mapsto g^3 L,
\end{align*}
where $g$ is such that $gL = L^\perp$. To show that it exists, we need to prove its independence of the choice of $g$: if $L \in \mathfrak{U}_{j/2}(S)$ and $g$, $g'$ are such that $gL = L^\perp$ and $g' L = L^\perp$, then $(g')^3 L = g^3 L$.

Indeed, for $f = g'g^{-1} \in Γ_{0/2}(S)$ we have $fg L = g'L = L^\perp = gL$, which implies $f^3 gL = gL$ and hence $(g')^3 L = g^2 f^3 g L = g^3 L$, where we used the commutativity of the group $Γ_\bullet (S)$.
\end{proof}

\begin{prop}
\label{indep}
Schwarz's extended Mumford form $M(g,L)$ is independent of $g$ and is, in fact, the pullback of a section $M(L)$ of the line bundle $\hat{γ}^* \Ber_{(-5j+3)/2} \otimes \big(\Ber^*_{j/2}\big)^5$, if $j$ is even, or its dual, if $j$ is odd,  on $\mathfrak{U}_{j/2}$:
\begin{equation*}
M(g,L) = p_2^* M(L).
\end{equation*}
\end{prop}

Having this statement in mind, by a slight abuse of terminology, we will not make a distinction between $M(L)$ and $M(g,L)$ and apply the term \emph{Schwarz's extended Mumford form} to either, depending on the context.

\begin{proof}
    One can interpret the definition of Schwarz's extended Mumford form as a section of the line bundle $  γ^* \Ber_{(-5j+3)/2} \otimes \big(\Ber^*_{j/2}\big)^5$ on $\tilde{\mathfrak{U}}_{j/2}$.
    What we want to show is that $M(L)$ is independent of $g$.
    If $g' \in Γ_{(1-2j)/2}$ is another element such that $g'L = L^\perp = gL$, then for $ f = g'g^{-1}$, by \eqref{taufg}, we have
    \begin{equation*}
    \frac{{\tau}_{g'L}((g')^2)}{{\tau}^2_{L}(g')}  = \frac{\tau_{gL}^2(f) \tau_{gL}(g^2) }{\tau_{gL}^2 (f){\tau}^2_{L}(g)} = \frac{\tau_{gL}(g^2)}{{\tau}^2_{L}(g)},
    \end{equation*}
    because $(g')^2 = (fg)^2 = g^2 f^2$, $f^2 gL = gL$, and
    \begin{equation*}
    {\tau}_{g'L}((g')^2) = {\tau}_{gL}((g')^2) = \tau_{g L}(g^2)  {\tau}_{gL}(f^2) = \tau_{g L}(g^2)  {\tau}_{gL}^2 (f).
    \qedhere
    \end{equation*}
\end{proof}

At last, we may prove the inavriance of Schwarz's extended super Mumford form under the super Heisenberg action and the Neveu-Schwarz action.

\begin{prop}\label{Gamma mumford invariance}
    Schwarz's extended Mumford form $M(L)$ defined over Schwarz's locus in \eqref{SchwarzMumfordform} is in\-vari\-ant under the action of the super Heisenberg group $Γ_{\bullet}$.
\end{prop}
\begin{proof}
    This result follows from the definition of $M(L)$ as a combination of Schwarz's super tau functions, which are each individually invariant under the action of $Γ_\bullet$ by \cref{Gamma action on tau}.
\end{proof}

\begin{theorem}\label{NS mumford invariance}
    Schwarz's extended Mumford form $M(L)$ defined over Schwarz's locus in \eqref{SchwarzMumfordform}  is in\-vari\-ant under the action of the Neveu-Schwarz group $\NS$.
\end{theorem}
\begin{proof}
    Applying the definition of Schwarz's extended form and factoring the numerator using \cref{taufg} gives:
    \begin{align*}
        M(L) = \frac{τ_{g^2L}(g) \; τ_{gL}(g)}{τ^2_L(g)},
    \end{align*}
    to which the action of $\tilde{φ}\in \NS$ applies as:
    \begin{align*}
        (\tilde{φ} M)(L) &= \frac{(\tilde{φ}{\tau})_{g^2L}(g) \; (\tilde{φ}{\tau})_{gL}(g)}{(\tilde{φ}{\tau})^2_{L}(g)}.
    \end{align*}
      The formula for Neveu-Schwarz action on tau functions is given in \cref{NS action on tau}. Since the value of $δ$ is preserved under the action of $Γ_\bullet$, the value of $δ$ determined by $L$, $gL$, and $g^2L$ are the same, and therefore the cocycles are all $C^δ$ in the resulting formula:
    \begin{align*}
        (\tilde{φ} M)(L) & =
        \left(   \frac{C^δ \left(φ,\Ad_{φ^{-1}}(g^{-1})\right)}{C^δ\left(g^{-1},φ\right)}    \right)^2 \left(\frac{C^δ\left(g^{-1},φ\right)}{C^δ \left(φ,\Ad_{φ^{-1}}(g^{-1})\right)} \right)^2 \cdot \frac{τ_{g^2L}(g) \; τ_{gL}(g)}{{τ}^2_{L}(g)}.
    \end{align*}
    Since the cocycles are otherwise matching, they simply cancel in pairs, which shows the invariance under the action of $\NS$.
\end{proof}


\color{black}

\section{The super Mumford form and Schwarz's extended Mumford form}

For the sake of completeness, it would be good to relate Schwarz's extended Mumford form to the super Mumford on the moduli space $\mathfrak{M}_{g}$ of genus $g$ super Riemann surfaces. Below is a reminder of how the super Mumford isomorphism and form are constructed explicitly, see \cite{Voronov.1988.afftMmist}. More details and a generalization to the punctured case may be found in \cite{Diroff.2019.tsMfitpoRaNSp}. The generalization to rational (meromorphic) sections below is new and will be useful for the proof of \cref{sMform}.

Suppose we have a smooth, proper family $π \co X \to S$ of genus $g$ super Riemann surfaces with $X$ being quasi-projective. This family represents an $S$-point of the supermoduli stack $\mathfrak{M}_{g}$.

\begin{definition}[\cite{Deligne.1988.,Voronov.1988.afftMmist}]
    Let $\mathcal{F}$ be a locally free sheaf on $X$. Then the \emph{Berezinian of cohomology of $\mathcal{F}$} is an invertible sheaf on $S$ given by
    \begin{align*}
        B(\mathcal{F}) \coeq \otimes_i\left(\Ber R^i p_*\mathcal{F}\right)^{(-1)^i},
    \end{align*}
    provided the higher direct images are locally free; otherwise, $B(\mathcal{F})$ may be generalized under these assumptions.
    We define the \emph{Berezinian line bundles $λ_{j/2}$} for the family $π\co X\to S $ as
    \begin{align*}
        \lambda_{j/2}\coeq B(ω _{X/S}^{\otimes j}),
    \end{align*}
    where $ω_{X/S }$ is the relative dualizing sheaf.
\end{definition}

\begin{theorem}[\cite{Deligne.1988.,Voronov.1988.afftMmist}]
\label{sMumfordIso}
    Under the above assumptions, there is a canonical isomorphism of line bundles on $S$, called the \emph{super Mumford isomorphism}:
    \begin{equation*}
    \lambda_{3/2} = \lambda_{1/2}^{5}.
    \end{equation*}
    Moreover, this isomorphism may be given by an explicit formula.
\end{theorem}

\begin{proof}
    Working locally on $S$, if we take a global odd rational  section $s$ of the odd line bundle $ω_{X/S}$ such that the reduction $s_\red$ is nonzero, then the Berezinian of $s$ acting on cohomology defines an invertible global regular section $B_0 (s)$ of $\lambda_{1/2}^{-1} \otimes \lambda_{0/2}^{-1} \otimes \mc{N}$, where $\mc{N}$ is a certain line bundle\footnote{$\mc{N}$ is, up to the parity change $Π^{g-1}$, the Berezinian of cohomology of the virtual coherent sheaf on $X$ obtained by restricting $ω_{X/S}$ to the divisor of $s$, the word virtual referring to taking linear combinations of the restrictions with multiplicities prescribed by the divisor.} on $S$, as well as a similar section $B_1 (s)$ of $\lambda_{2/2} \otimes \lambda_{1/2} \otimes \mc{N}$ and a similar section $B_2 (s)$ of $\lambda_{3/2}^{-1} \otimes \lambda_{2/2}^{-1} \otimes \mc{N}$.
    These sections come from the following construction.

    If $s$ is an even regular section  of an even line bundle $\mc{L}$ with a nonzero $s_\red$ such that $\div1 s_\red$ has only simple zeroes, then the short exact sequence of sheaves
    \begin{equation}
        \label{SES}
        0 \to \OO_X \xrightarrow{s} \mc{L} \to \mc{L}|_{\div1 s} \to 0
    \end{equation}
    yields an isomorphism
    \begin{equation}
        \label{Bs}
        B_{\OO} (s)\co B(\OO_X) \otimes B(\mc{L}|_{\div1 s}) \to B(\mc{L}).
    \end{equation}

    If $s$ is rational, then the same story works due to the following trick. Let $\div1 s = D_1 - D_2$, where $D_1$ and $D_2$ are the divisors of zeros and poles of $s$, respectively. Let us assume the reductions of these divisors are simple for the time being. Let $s_1$ and $s_2$ be the canonical sections of $\OO_X (D_1)$ and $\OO_X (D_2)$, respectively, chosen in such a manner that $s = s_1 / s_2$. Note that $\mc{L} = \OO_X (D_1 - D_2)$. As before, the short exact sequences
    \begin{equation*}
    0 \to \OO_X (-D_2) \xrightarrow{s_2} \OO_X \to \OO_X |_{D_2} \to 0
    \end{equation*}
    and
    \begin{equation*}
    0 \to \OO_X (-D_2) \xrightarrow{s_1} \OO_X (D_1-D_2) \to \OO_X (D_1-D_2)|_{D_1} \to 0
    \end{equation*}
    yield isomorphisms
    \begin{equation*}
    B_{-D_2}(s_2)\co B(\OO_X(-D_2)) \otimes B(\OO_X |_{D_2}) \to B(\OO_X)
    \end{equation*}
    and
    \begin{equation*}
    B_{-D_2}(s_1)\co B(\OO_X(-D_2)) \otimes B(\OO_X (D_1 - D_2) |_{D_1}) \to B(\OO_X (D_1 - D_2)),
    \end{equation*}
    respectively. Then the section $B_{-D_2}(s_1)/B_{-D_2}(s_2)$ gives an isomorphism
    \begin{equation*}
    B(\OO_X) \otimes B(\OO_X |_{D_2})^{-1} \otimes B(\OO_X (D_1 - D_2) |_{D_1}) \to B(\OO_X (D_1 - D_2)),
    \end{equation*}
    which we denote by $B_0(s)$. In \cite{Voronov.1988.afftMmist}, the following principle was established:
    \begin{gather}
        \label{principleS}
        \text{The $B(\MM|_{D})$'s are canonically isomorphic for all even line}\\
        \nonumber
        \text{bundles $\MM$ and a fixed effective divisor $D$ with no multiplicities.}
    \end{gather}
    Thus, $B(\OO_X |_{D_2}) = B(\OO_X (D_1 - D_2) |_{D_2})$ and, therefore, $B(\OO_X |_{D_2})^{-1} \otimes B(\OO_X (D_1 - D_2) |_{D_1}) \linebreak[0] = B(\OO_X (D_1 - D_2) |_{D_2})^{-1} \otimes B(\OO_X (D_1 - D_2) |_{D_1})$, which we can combine into
    \begin{equation*}
    B(\OO_X (D_1 - D_2) |_{D_1 - D_2}) \coeq B(\OO_X (D_1 - D_2) |_{D_2})^{-1} \otimes B(\OO_X (D_1 - D_2) |_{D_1}).
    \end{equation*}
    Recalling that $\mc{L} = \OO_X (D_1 - D_2)$ and $\div1 s = D_1 - D_2$, we get the same isomorphism \eqref{Bs} for $s$ being rational with simple zeros and poles. If $\div1 s$ has multiplicities, i.e., $\div1 s = \sum_P n_P P$, $P$ being prime divisors and $n_P \in \ZZ$, then we will have to interpret $\mc{L}|_{\div1 s}$ as a virtual coherent sheaf $\sum_P n_P \mc{L}|_P$ and the factor $B(\mc{L}|_{\div1 s})$ will be identified as $\bigotimes_P B(\mc{L}|_P)^{\otimes n_P}$. In reality, each $n_P \mc{L}|_P$'s for $\abs{n_P} > 1$ arises as an extension by $\mc{L}|_P$ tensored with powers of the conormal bundle of the divisor $P$, but these factors may be ignored due to Principle \eqref{principleS} and extensions may be converted to direct sums because of the multiplicativity of the Berezinian of cohomology functor $B(-)$.

    If $s$ is an odd rational section of an odd line bundle $\mc{L}$ with a nonzero $s_\red$, then the same argument yields an isomorphism
    \begin{equation*}
    B_\OO (s)\co B(\OO_X) \otimes B(\mc{L}|_{\div1 s})^{-1} \to Π^{g-1} B(\mc{L})^{-1}.
    \end{equation*}
    Here $g-1$ is, by super Riemann-Roch, the super Euler characteristic of $\mc{L}$: $sχ(\mc{L}) = \sdim H^0(\mc{L}) - \sdim H^1(\mc{L}) = g-1$, where the superdimension $\sdim$ is the difference between the even and odd dimensions. Thus, $B_0 (s)$ is a trivializing section of $Π^{g-1} B(\mc{L})^{-1} \otimes B(\OO_X)^{-1} \otimes  B(\mc{L}|_{\div1 s})$. If we tensor the short exact sequence \Cref{SES} by a line bundle $\MM$, we will get, for an odd $\mc{L}$, an isomorphism
    \begin{align}\label{B(s)}
    	B_{\MM} (s)\co
        \begin{cases}
            B(\MM) \otimes B((\mc{L} \otimes \MM)|_{\div1 s})^{-1} \to Π^{g-1} B(\mc{L} \otimes \MM)^{-1} & \text{for }\mathcal{M} \text{ even},\\[0.4em]
            Π^{g-1} B(\MM)^{-1} \otimes B((\mc{L} \otimes \MM)|_{\div1 s}) \to  B(\mc{L} \otimes \MM) & \text{for }\mathcal{M} \text{ odd}.
        \end{cases}
    \end{align}

    Taking into account Principle \eqref{principleS} along with the isomorphism
    \begin{equation*}
    B(Π \MM|_D) = B(\MM|_D)^{-1},
    \end{equation*}
    observe that for $\mathcal{L}$ odd, the section $B_{\MM} (s) / B_\OO (s)$ provides an isomorphism
    \begin{align*}
    B_{\MM} (s) / B_\OO (s)\co
    \begin{cases}
    	B(\mc{L})^{-1} \otimes B(\MM) \otimes B(\OO_X)^{-1} \to  B(\mc{L} \otimes \MM)^{-1} & \text{for }\mathcal{M} \text{ even},\\[0.4em]
        B(\mc{L})^{-1} \otimes B(\MM)^{-1} \otimes B(\OO_X)^{-1} \to  B(\mc{L} \otimes \MM) & \text{for }\mathcal{M} \text{ odd}.
    \end{cases}
\end{align*}

    If we use $ω_{X/S}$ for $\mc{L} $, $ω_{X/S}$ and $ω_{X/S}^2$ for $\MM$, and $B_j(s)$ for $B_{ω^j}(s)$,
    we conclude that the section $B_0 (s)/B_1 (s)$
    gives a trivializing section of
    \begin{equation*}
    \lambda_{2/2}^{-1} \otimes \lambda_{1/2}^{-2} \otimes \lambda_{0/2}^{-1}.
    \end{equation*}
    Similarly, $B_0 (s) / B_2 (s)$ gives a trivializing section of
    \begin{equation*}
    \lambda_{3/2} \otimes \lambda_{2/2} \otimes \lambda_{1/2}^{-1} \otimes \lambda_{0/2}^{-1}.
    \end{equation*}
    Hence, the product
    \begin{equation}
        \label{MumfordFormula}
        \mu(s) \coeq \frac{B_0 (s)}{B_1 (s)} \cdot \frac{B_0 (s)}{B_2 (s)}
    \end{equation}
    produces a trivializing section of
    \begin{equation*}
    \lambda_{3/2} \otimes \lambda_{1/2}^{-3} \otimes \lambda_{0/2}^{-2}.
    \end{equation*}
    Finally, combining this with Serre duality
    \begin{equation*}
    \lambda_{1/2} = \lambda_{0/2},
    \end{equation*}
    we get a trivializing section $\mu(s)$ of
    \begin{equation*}
    \lambda_{3/2} \otimes \lambda_{1/2}^{-5}.
    \end{equation*}
    The section $\mu = \mu(s)$ does not depend on $s$. The argument is similar to the one used to prove \Cref{indep}. This canonical section is called the \emph{super Mumford form}.
\end{proof}

\begin{theorem}
    \label{sMform}
    For $j = 0$, Schwarz's extended Mumford form, defined in \eqref{SchwarzMumfordform}, restricts to the usual super Mumford form on $\mathfrak{M}_{g,1_\NS^\infty}$, where the moduli space of SRS is embedded into Schwarz's locus in the super Sato Grassmannian by the 0\textsuperscript{th} Krichever map $\kappa_{0/2} \co \mathfrak{M}_{g,1_\NS^\infty} \to \mathfrak{U}_{0/2} \subset \Gr_{0/2}$:
    \begin{align*}
        \kappa_{0/2}^* M =  \mu.
    \end{align*}
\end{theorem}
Here the super Mumford form $\mu$ on $\mathfrak{M}_{g,1_\NS^\infty}$ is just the pullback of the super Mumford form on $\mathfrak{M}_{g}$ via the map $\mathfrak{M}_{g,1_\NS^\infty} \to \mathfrak{M}_g$ forgetting the puncture and the formal coordinate at it.
\begin{proof}
    First, let us relate the sections $B_j(s)$ from the construction of \Cref{MumfordFormula} with Schwarz's super tau function. For $j \in \ZZ$, the Berezinian line bundle $\lambda_{j/2}$ on the moduli space $\mathfrak{M} = \mathfrak{M}_{g,1_\NS^\infty}$ is defined as the Berezinian of cohomology of the $j$\textsuperscript{th} power $ω^j_{X/\mathfrak{M}}$ of the relative dualizing sheaf of the universal SRS $X \to \mathfrak{M}$. This cohomology, locally on $\mathfrak{M}$, may be computed via the Čech complex
    \begin{equation*}
        0 \to   ω^j_{X/\mathfrak{M}} (U) \oplus ω^j_{X/\mathfrak{M}}(V) \to ω^j_{X/\mathfrak{M}}(U\cap V) \to  0
    \end{equation*}
    for $U$ being the complement to the puncture in the family of underlying Riemann surfaces and $V$ being a neighborhood of the puncture. On the moduli space $\mathfrak{M}$, we can use the even formal coordinate $z$ near the puncture and pass to completion in the $z$-adic topology, which will transform the above Čech complex to the familiar complex, cf.\ \Cref{Ber-l-b},
    \begin{equation*}
        0 \to   L \oplus H_{j/2}^+ \to H_{j/2} \to  0
    \end{equation*}
    without affecting cohomology. Here $L \coeq ω^j_{X/\mathfrak{M}}(U)$. Now, given a family $π\co X \to S$ of SRSs with а puncture, formal coordinates at the puncture, and a global rational section $s$ of $ω_{X/S}$, i.e., a section of $π_* ω_{X/S}$ rational along the fibers of $π$,  with $s_{\red} \ne 0$ we can write the restriction of $s$ to $U \cap V$ in the given formal coordinates $(z | ζ)$, which will result in a formal Laurent series $ s = s(z | ζ) [dz | dζ] \in Γ_{1/2}(S)$ after completion. If $s$ happens to be regular, then it acts on the Berezinian of cohomology via
    \begin{align} \label{B(s)j}
        B_j(s)\co
        \begin{cases}
        	B(ω^j_{X/S}) \otimes B(ω^{j+1}_{X/S}|_{\div1 s})^{-1}\to Π^{g-1} B(ω^{j+1}_{X/S})^{-1} & \text{for } j \text{ even},\\[0.8em]
            Π^{g-1} B(ω^j_{X/S})^{-1} \otimes B(ω^{j+1}_{X/S}|_{\div1 s})\to B(ω^{j+1}_{X/S}) &  \text{for }j \text{ odd},
        \end{cases}
    \end{align}
    see \cref{B(s)}. If $s$ is rational along the fibers of $π\co X \to S$, then the argument below may be easily generalized using the argument after \eqref{Bs}, so we will concentrate on the case when $s$ is regular. In the following lemma, let $p_1$ denote projection to the first factor.

    \begin{lemma}
        For a global section $s$ of $π_* ω_{X/S}$ with $s_\red \ne 0$, regarded as an element of $s \in Γ_{1/2}(S)$, we have canonical isomorphisms of line bundles
        \begin{equation*}
        \kappa_{j/2}^* \Ber_{j/2} = \lambda_{j/2},
        \end{equation*}
        where $\kappa_{j/2} \co S \to \Gr_{j/2}$ is the super Krichever map,  and
        \begin{equation}
            \label{Berandlambda}
            (\kappa_{j/2}, \id_S)^* s^* p_1^* \Ber_{(j+1)/2} = B(ω_{X/S}^{j+1}|_{\div1 s})^{-1} \otimes \lambda_{(j+1)/2},
        \end{equation}
        where the map $s \co \Gr_{j/2} \times S \to \Gr_{(j+1)/2} \times S$ is the $Γ_{1/2}(S)$-action map. Moreover,
        the Berezinian $B_j(s) \linebreak[1] \in \linebreak[0] H^0 (S, \lambda_{j/2} \otimes B(ω_{X/S}^{j+1}|_{\div1 s})^{-1} \otimes \lambda_{(j+1)/2})$ and Schwarz's super tau function  $\tau_{sL} (s^{-1}) \linebreak[2] \in \linebreak[1] H^0 (\Gr_{j/2} \times S, \linebreak[0] p_1^* \Ber_{j/2} \otimes s^* p_1^* \Ber_{(j+1)/2})$, are related as follows:
        \begin{equation}
            \label{Bandtau}
            B_j (s) = (\kappa_{j/2}, \id_S)^* \tau_{sL} (s^{-1}) = (\kappa_{j/2}, \id_S)^* \tau_{L}^{-1} (s).
        \end{equation}
    \end{lemma}

    \begin{proof}[Proof of Lemma]
        See \cite{Maxwell.2022.tsMfaSG} regarding the relation between the Berezian line bundles on the super Sato Grassmannian and supermoduli space under the super Krichever map.

        Let us prove the second isomorphism and the relation between $B_0(s)$ and Schwarz's super tau function.

        For a family $π \co X \to S$ of SRSs with one NS puncture and formal coordinates $(z | ζ)$ at the puncture, the short exact sequence \eqref{SES} tensored with $ω_{X/S}$ induces a short exact sequence of Čech complexes
        \begin{equation*}
            \begin{tikzcd}
                &  0 \arrow{d} & 0 \arrow{d}\\
                0 \arrow{r} &  ω^j_{X/S}(U) \oplus ω^j_{X/S}(V) \arrow{r}\arrow{d}{s}& ω^j_{X/S}(U\cap V) \arrow{d}{s} \arrow{r} & 0\\
                0  \arrow{r} & ω^{j+1}_{X/S} (U) \oplus ω^{j+1}_{X/S} (V) \arrow{r}\arrow{d} & ω^{j+1}_{X/S} (U\cap V) \arrow{r}\arrow{d} & 0\\
                0  \arrow{r} & \left.ω^{j+1}_{X/S}\right|_{\div1 s} (U) \arrow{r}\arrow{d} & 0 \arrow{r}\arrow{d} & 0,\\
                &  0  & 0
            \end{tikzcd}
        \end{equation*}
        provided $V$ is small enough to not intersect $\div1 s$. On the base $S$ of the family, this diagram rewrites, after the $z$-adic completion, as
        \begin{equation*}
            \begin{tikzcd}
                &  0 \arrow{d} & 0 \arrow{d}\\
                0 \arrow{r} &  L \oplus H_{j/2}^+ \arrow{r}\arrow{d}{s}& H_{j/2} \arrow{d}{s} \arrow{r} & 0\\
                0  \arrow{r} &  L' \oplus H_{(j+1)/2}^+ \arrow{r} \arrow{d} & H_{(j+1)/2} \arrow{r} \arrow{d} & 0\\
                0  \arrow{r} & L'/ s L \arrow{r}\arrow{d} & 0 \arrow{r}\arrow{d} & 0,\\
                &  0  & 0
            \end{tikzcd}
        \end{equation*}
        where $L' \coeq ω^{j+1}_{X/S}(U)$,  $s = s(z | ζ) [dz | dζ]$ is regarded as an element of $Γ_{1/2}(S)$. Observe that this short exact sequence of complexes is quasi-isomorphic to
        \begin{equation*}
            \begin{tikzcd}
                &  0 \arrow{d} & 0 \arrow{d}\\
                0 \arrow{r} &  L \oplus H_{j/2}^+ \arrow{r}\arrow{d}{s}& H_{j/2} \arrow{d}{s} \arrow{r} & 0\\
                0  \arrow{r} &  sL \oplus H_{(j+1)/2}^+ \arrow{r} \arrow{d} & H_{(j+1)/2} \arrow{r} \arrow{d} & 0.\\
                &  0  & 0
            \end{tikzcd}
        \end{equation*}
        Therefore, the isomorphism \eqref{B(s)j}, which can now be rewritten as
        \begin{align*}
            B_j(s)\co
            \begin{cases}
                \Ber_{j/2}(L) \otimes \Ber(L'/ s L)^{-1}\to Π^{g-1} \Ber_{(j+1)/2}(L')^{-1} &  \text{for }j \text{ even},\\[0.6em]
            Π^{g-1} \Ber_{j/2}(L)^{-1} \otimes \Ber(L'/ s L)\to  \Ber_{(j+1)/2}(L') & \text{for }j \text{ odd},
            \end{cases}
        \end{align*}
        turns into the action of $s \in Γ_{1/2}(S)$,
        \begin{align*}
            \tau_{sL}(s^{-1}) = \Ber s\co
            \begin{cases}
            	 \Ber_{j/2}(L)\to Π^{g-1} \Ber_{(j+1)/2}(s L)^{-1}  & \text{for } j \text{ even},\\[0.6em]
             Π^{g-1} \Ber_{j/2}(L)^{-1} \to \Ber_{(j+1)/2}(s L) & \text{for } j \text{ odd},
            \end{cases}
        \end{align*}
        on the Berezinian line bundles described in \cref{Gamma-action}. Geometrically, this observation implies \cref{Bandtau}, and the canonical isomorphism
        \begin{align*}
            \Ber_{(j+1)/2}(s L) = \Ber(L'/ s L)^{-1} \otimes \Ber_{(j+1)/2}(L')
        \end{align*}
        implies \eqref{Berandlambda}.

        Lemma is proven.
        \let\qed\relax

    \end{proof}

    To finish the proof of the theorem, plug \Cref{Bandtau} into the formula \eqref{MumfordFormula}:
    \begin{multline*}
        \mu(s) = \frac{B_0 (s)}{B_1 (s)} \cdot \frac{B_0 (s)}{B_2 (s)} = \frac{\kappa^*_{0/2} \tau_L^{-1} (s)}{\kappa^*_{1/2}\tau_{sL}^{-1} (s)} \cdot \frac{\kappa^*_{0/2}\tau_L^{-1} (s)}{\kappa^*_{2/2}\tau_{s^2 L}^{-1} (s)}
        =
        \frac{\kappa^*_{1/2}\tau_{sL} (s)}{\kappa^*_{0/2} \tau_L (s)} \cdot \frac{\kappa^*_{2/2}\tau_{s^2 L} (s)}{\kappa^*_{0/2}\tau_L (s)}
        \\
        = \frac{\kappa^*_{0/2}\tau_{L} (s)}{\kappa^*_{0/2}\tau_L (s)} \cdot \frac{\kappa^*_{1/2}\tau_{sL} (s)}{\kappa^*_{0/2} \tau_L (s)} \cdot \frac{\kappa^*_{2/2}\tau_{s^2 L} (s)}{\kappa^*_{0/2}\tau_L (s)}  = \frac{\kappa^*_{0/2}\tau_{L} (s^3)}{\kappa^*_{0/2}\tau^3_L (s)} = \kappa^*_{0/2} M (s,L)
    \end{multline*}
    by \eqref{SchwarzMumfordform}.
\end{proof}

\section*{Data Availability}

We do not analyze or generate any datasets because our work proceeds within an abstract and formulaic approach.

\printbibliography

\end{document}